\documentclass[a4paper, 12pt]{amsart}
\usepackage[T1]{fontenc}
\usepackage{xspace}
\usepackage{enumerate}
\usepackage[utf8]{inputenc} 
\usepackage[english]{babel}
\usepackage{lmodern}
\usepackage{microtype}
\usepackage{amssymb}
\usepackage{verbatim}
\usepackage{amsthm}
\usepackage{amsmath}
\usepackage{mathtools}
\usepackage{color}
\usepackage{tikz}
\usepackage{threeparttable}
\usepackage{graphicx}
\usepackage{caption}
\usepackage{paralist}
\usepackage[pagebackref=true, pdfborder={0 0 0}, linktocpage]{hyperref}
\usepackage{backref}
\usepackage{aliascnt}
\usepackage{cleveref}
\usepackage{dsfont}
\usepackage{mathrsfs}

\newtheorem{theorem}{Theorem}[section]
\newtheorem{lemma}[theorem]{Lemma}

\newtheorem{corollary}[theorem]{Corollary}
\newtheorem{algo}{Algorithm}[section]

\theoremstyle{remark}
\newtheorem{remark}[theorem]{Remark}

\begin{document}

\title[]
{On the continuous time limit of the Ensemble Kalman Filter}
\author{Theresa Lange, Wilhelm Stannat}
\address{%
Institut f\"ur Mathematik\\
Technische Universit\"at Berlin \\
Stra{\ss}e des 17. Juni 136\\
D-10623 Berlin\\
and\\
Bernstein Center for Computational Neuroscience\\
Philippstr. 13\\
D-10115 Berlin\\
Germany}
\email{tlange@math.tu-berlin.de, stannat@math.tu-berlin.de}

\begin{abstract}
We present recent results on the existence of a continuous time limit for Ensemble Kalman Filter algorithms. In the setting of continuous signal and observation processes, we apply the original Ensemble Kalman Filter algorithm proposed by \cite{burgers1998} as well as a recent variant \cite{deWiljes2017} to the respective discretizations and show that in the limit of decreasing stepsize the filter equations converge to an ensemble of interacting (stochastic) differential equations in the ensemble-mean-square sense. Our analysis also allows for the derivation of convergence rates with respect to the stepsize.\\
An application of our analysis is the rigorous derivation of continuous ensemble filtering algorithms consistent with discrete approximation schemes. Conversely, the continuous time limit allows for a better qualitative and quantitative analysis of the time-discrete counterparts using the rich theory of dynamical systems in continuous time.
\end{abstract}


\maketitle

\section{Introduction} 

\noindent 
In this paper, we aim to give a rigorous derivation of a continuous time limit of the Ensemble Kalman Filter. The Ensemble Kalman Filter (EnKF), or Ensemble Kalman-Bucy Filter (EnKBF) when considered in continuous time, is a data assimilation technique which since its invention in the 1990s gained wide popularity in many scientific fields such as oceanography or meteorology. The general idea is the following: a $d$-dimensional Markovian signal $X$ described by
\begin{equation}\label{model}
{\rm d}X_t = f\left(X_t\right){\rm d}t + Q^{\frac{1}{2}}{\rm d}W_t,
\end{equation}
is unknown but can be observed through a $p$-dimensional process $Y$
\begin{equation}\label{obs}
{\rm d}Y_t = g\left(X_t\right){\rm d}t + C^{\frac{1}{2}}{\rm d}V_t
\end{equation}
where $W$ and $V$ are independent Brownian motions. Based on the model assumptions and the real-time measurements, a filter calculates an estimate of the signal at the current time $t$. In case of the EnK(B)F, the signal is estimated in a Monte-Carlo fashion, i.e. an ensemble of initial conditions $X_0^{(i)}$, $i=1,...,M$, is propagated according to (\ref{model}) and (\ref{obs}) such that its ensemble mean
\begin{equation}
\bar{x}_t := \frac{1}{M} \sum_{i=1}^M X_t^{(i)}
\end{equation}
yields an estimate for the true signal at time $t$.

\smallskip 
\noindent 
In \cite{kelly2014}, the authors give a continuous-time formulation of these equations in the case of linear observations which we generalize as follows:
\begin{align}\label{classical}
{\rm d}X_t^{(i)} = f\left(X_t^{(i)}\right){\rm d}t &+ Q^{\frac{1}{2}}{\rm d}W_t^{(i)} \notag\\
&+ \frac{1}{M-1}E_t \mathcal{G}_t^T C^{-1} \left( {\rm d}Y_t + C^{\frac{1}{2}}{\rm d}V_t^{(i)} - g\left(X_t^{(i)}\right){\rm d}t\right),
\end{align}
where
\begin{align} 
E_t &:= \left[X_t^{(i)}-\bar{x}_t\right]_{i=1,...,M} \in \mathds{R}^{d \times M},\label{Et}\\
\mathcal{G}_t &:= \left[g\left(X_t^{(i)}\right)-\bar{g}_t\right]_{i=1,...,M} \in \mathds{R}^{p \times M}\label{Gt}
\end{align}
with
\begin{equation}
\bar{g}_t := \frac{1}{M}\sum_{i=1}^M g\left(X_t^{(i)}\right).
\end{equation}

\medskip 
\noindent 
The continuous formulation is of great use when investigating the mathematical properties of the filter by means of the rich theory of stochastic differential equations and continuous time dynamical systems. In the numerical application, however, only the discrete algorithm can be implemented. To relate both formulations to each other one analyzes the continuous time limit whose existence implies that the continuous formulation is an intrinsic result of the discrete filter. This means that each property we derive for the continuous formulation is also a property of the discrete scheme independent of the discretization step.\\
The above continuous formulation was achieved by considering the classical EnKF algorithm (cf. \cite{burgers1998}) and rearranging terms such that the result resembled an Euler-Maruyama scheme of a stochastic differential equation yielding (\ref{classical}). A similar solution can be found in \cite{schillings2017} in the context of inverse problems where the authors derive a tamed Euler-Maruyama type discretization using an EnKF algorithm given in \cite{iglesias2013}.

\smallskip
\noindent
As known from classical results, the Euler-Maruyama scheme strongly converges to the solution of the underlying stochastic differential equation under some regularity assumptions on the coefficients (cf. \cite{kloeden1992}).\\
Nevertheless, a rigorous analysis of the convergence of the Euler-Maruyama discretization is still required to verify the afore mentioned consistency. A first attempt to rigorously show a continuous time limit for the EnKF in the context of inverse problems is given in \cite{bloemker2018} but only for a simplified example. In this work, we investigate the limit for a more general setting and explicitly show convergence to (\ref{classical}) with effective rates in terms of the stepsize in Theorem \ref{ClassicalLimit} and Corollary \ref{ImprovedRates} for explicit rates.

\smallskip 
\noindent 
It turns out that assuring the additional integrability condition for the explicit rates is not that easy. The reason for this originates in the stochastic perturbation of the observation in the update step of the filter. In \cite{burgers1998}, the authors showed that omitting this perturbation yields an underestimation of the resulting covariance of the ensemble. Since this additional source of noise complicates the attempts of achieving better rates, one may consider different filtering approaches avoiding this particular step.

\smallskip
\noindent 
For instance consider the so called deterministic filtering algorithms (such as the Ensemble Square Root Filters (EnSRF), see e.g. \cite{whitaker2002} ) that aim at transforming the ensemble without additional noise and in such a way that the correct ensemble statistics are obtained. Recently in \cite{deWiljes2017}, a modified version of (\ref{classical}) was analyzed:
\begin{align}\label{modified}
{\rm d}X_t^{(i)} = &f\left(X_t^{(i)}\right){\rm d}t + \frac{1}{2}Q P_t^{-1}\left(X_t^{(i)} - \bar{x}_t\right){\rm d}t\notag\\
&+ \frac{1}{M-1}E_t \mathcal{G}_t^T C^{-1} \left({\rm d}Y_t + \frac{1}{2} \left(g\left(X_t^{(i)}\right)-\bar{g}_t\right){\rm d}t- g\left(X_t^{(i)}\right){\rm d}t\right)
\end{align}
with the ensemble covariance matrix
\begin{equation}
P_t := \frac{1}{M-1} \sum_{i=1}^M \left(X_t^{(i)}-\bar{x}_t\right)\left(X_t^{(i)}-\bar{x}_t\right)^T.
\end{equation}
In this formulation which we will call \textit{modified formulation} as opposed to the \textit{classical formulation} (\ref{classical}), the noise terms are replaced by 
\begin{equation}\label{noiseReplacement}
\begin{aligned}
Q^{\frac{1}{2}}{\rm d}W_t^{(i)} &\leadsto \frac{1}{2}Q P_t^{-1}\left(X_t^{(i)} - \bar{x}_t\right){\rm d}t,\\
C^{\frac{1}{2}}{\rm d}V_t^{(i)} &\leadsto \frac{1}{2} \left(g\left(X_t^{(i)}\right)-\bar{g}_t\right){\rm d}t.
\end{aligned}
\end{equation}
Note that with this choice the evolution equations for the ensemble mean $\bar{x}_t$ and the covariance matrix $P_t$ read
\begin{equation}
{\rm d}\bar{x}_t = \bar{f}_t{\rm d}t + \frac{1}{M-1}E_t\mathcal{G}_t^TC^{-1}\left({\rm d}Y_t - \bar{g}_t{\rm d}t\right)
\end{equation}
and
\begin{equation}
\begin{aligned}
\frac{{\rm d}}{{\rm d}t} P_t &= \frac{1}{M-1} \sum_{i=1}^M \left( \left(f\left(X_t^{(i)}\right)-\bar{f}_t\right)\left(X_t^{(i)}-\bar{x}_t\right)^T\right.\\
&\hspace{3cm}\left. + \left(X_t^{(i)}-\bar{x}_t\right)\left(f\left(X_t^{(i)}\right)-\bar{f}_t\right)^T\right)\\
&\hspace{0.5cm} + Q - \frac{1}{(M-1)^2}E_t\mathcal{G}_t^T C^{-1} \mathcal{G}_t E_t^T,
\end{aligned}
\end{equation}
which in the linear case coincide with the Kalman-Bucy equations explaining the particular choice of perturbation.\\

\noindent
For (\ref{modified}), the authors of \cite{deWiljes2017} were able to prove long-time stability and accuracy properties of the EnKBF. We therefore also give a continuous time limit analysis for (\ref{modified}) which might be a first contribution especially to future discussions on how to treat the noise numerically in context of achieving better approximation results.

\medskip 
\noindent 
The paper is organized as follows: after specifying in Section \ref{setting} the setting we consider throughout this paper, we shortly review in Section \ref{EnKF} the EnKF algorithm for the above mentioned versions. In Section \ref{Results}, we turn to the continuous time limit analysis and discuss our results both on the classical and the modified formulation as given in Theorem \ref{ClassicalLimit} and Theorem \ref{ModifiedLimit}, respectively. The proofs will then be given in Section \ref{Proofs} where we use various properties of the continuous and the discrete time processes involved which we give in Appendix \ref{appendixVt} and Appendix \ref{appendixVk}, respectively.

\section{Setting}\label{setting}

\noindent 
Consider the $d$-dimensional signal process modeled by
\begin{equation}
{\rm d}X_t = f(X_t) {\rm d}t + Q^{\frac{1}{2}}{\rm d}W_t, \hspace{0.5cm} X_t \in \mathds{R}^d
\end{equation}
with drift $f: \mathds{R}^d \rightarrow \mathds{R}^d$, $Q \in \mathds{R}^{d\times d}$ symmetric, positive definite, and $W = (W_t)_{t \geq 0}$ a $d$-dimensional standard Brownian motion.\\
The signal is observed via a $p$-dimensional process $Y \subset \mathds{R}^p$ modeled by
\begin{equation}
{\rm d}Y_t = g(X_t){\rm d}t + C^{\frac{1}{2}}{\rm d}V_t, \hspace{0.5cm}Y_t \in \mathds{R}^{p}
\end{equation}
with forward map $g : \mathds{R}^d \rightarrow \mathds{R}^p$, $C \in \mathds{R}^{p\times p}$ symmetric, positive definite, and $V = (V_t)_{t \geq 0}$ a $p$-dimensional standard Brownian motion.\\
Further $X_0$, $W$ and $V$ are mutually independent. Throughout this paper we assume that both $f$ and $g$ are Lipschitz-continuous with Lipschitz-constants $\|f\|_{\text{Lip}}$ and $\|g\|_{\text{Lip}}$, and bounded by $\|f\|_{\infty}$ and $\|g\|_{\infty}$, respectively, with $\|\cdot\|_{\infty}$ the usual supremum norm, as well as that all random entities are defined on a given probability space $(\Omega, \mathcal{F}, \mathds{P})$.

\medskip 
\noindent 
In some time horizon $T$ we consider the partition
\[ 0 = t_0 < t_1 < \dots <t_L = T\]
with step-size $h > 0$, i.e. $t_{k+1} = t_k + h = (k+1)h$.\\
The Euler-Maruyama scheme yields a time-discretization of the signal and observation process as follows:
\begin{align}
X_{t_k} &= X_{t_{k-1}} + h f(X_{t_{k-1}}) + Q^{\frac{1}{2}}\left(W_{t_k}-W_{t_{k-1}}\right)\\
Y_{t_k} &= Y_{t_{k-1}} + hg(X_{t_{k-1}}) + C^{\frac{1}{2}}\left(V_{t_k}-V_{t_{k-1}}\right)\\
\Rightarrow \Delta Y_k &:= Y_{t_k} - Y_{t_{k-1}} = h g(X_{t_{k-1}})+ C^{\frac{1}{2}}\left(V_{t_k}-V_{t_{k-1}}\right).
\end{align}
Thus $\Delta Y$ is the discrete-time observation process for the above time-discretization of $X$.

\section{Ensemble Kalman Filter algorithms}\label{EnKF}

\subsection{The classical formulation}\label{ClassicalAlgo}

\noindent
We adopt the formulation in \cite{law2015} as follows: assume that at time $t_{k-1}$, we have an ensemble of $M$ analyzed ensemble members $\left(X_{t_{k-1}}^{(i),a}\right)_{i=1,...,M}$. The forecast step yields a new ensemble consisting of the forecasted members, i.e. for each $i = 1, ..., M$ obtain
\begin{equation}
X_{t_k}^{(i),f} = X_{t_{k-1}}^{(i),a} + hf\left(X_{t_{k-1}}^{(i),a}\right) + Q^{\frac{1}{2}}\tilde{W}_k^{(i)}
\end{equation}
where $\left(\tilde{W}_k^{(i)}\right)_{i=1,...,M}$ is an i.i.d. sequence of samples of $\mathcal{N}(0, h{\rm Id})$. The ensemble mean is then given by
\begin{equation}
\bar{x}_k^{f} = \frac{1}{M} \sum_{i=1}^M X_{t_k}^{(i),f}.
\end{equation}

\medskip 
\noindent 
In the update step we stochastically perturb the new observation, i.e. define
\begin{equation}
\Delta Y_k^{(i)} := \Delta Y_k + C^{\frac{1}{2}}\tilde{V}_k^{(i)},
\end{equation}
where $\left(\tilde{V}_k^{(i)}\right)_{i=1,...,M}$ is an i.i.d. sequence of samples of $\mathcal{N}(0,h{\rm Id})$, and formulate an analyzed ensemble
\begin{equation}\label{classicalUpdate}
X_{t_k}^{(i),a} = X_{t_k}^{(i),f} + K_k \left( \Delta Y_k^{(i)} - hg\left(X_{t_k}^{(i),f}\right)\right)
\end{equation}
with Kalman gain matrix
\begin{equation}\label{KalmanGain}
K_k = \frac{1}{M-1}E_k^f \left(\mathcal{G}_k^f\right)^T \left(C + \frac{h}{M-1} \mathcal{G}_k^f \left(\mathcal{G}_k^f\right)^T\right)^{-1}
\end{equation}
where 
\begin{align}
E_k^{f} &= \left[X_{t_k}^{(i),f} - \bar{x}_k^{f}\right]_{i=1,...,M},\\
\mathcal{G}_k^{f} &= \left[g\left(X_{t_k}^{(i),f}\right)-\bar{g}_k^{f}\right]_{i=1,...,M},\hspace{1cm} \bar{g}_k^{f} &= \frac{1}{M} \sum_{i=1}^M g\left(X_{t_k}^{(i),f}\right).
\end{align}

\smallskip
\noindent
In total, the algorithm for the ensemble members as well as centered ensemble members reads:
\begin{algo}\label{classicalAlgo}
Forecast step:
\begin{equation}
\begin{aligned}
&X_{t_k}^{(i),f} = X_{t_{k-1}}^{(i),a} + hf\left(X_{t_{k-1}}^{(i),a}\right) + Q^{\frac{1}{2}}\tilde{W}_k^{(i)},\\
&X_{t_k}^{(i),f} - \bar{x}_k^{f} = X_{t_{k-1}}^{(i),a}-\bar{x}_{k-1}^{a} + h\left(f\left(X_{t_{k-1}}^{(i),a}\right)-\bar{f}_{k-1}^{a}\right) + Q^{\frac{1}{2}}\left(\tilde{W}_k^{(i)}-\bar{w}_k\right)
\end{aligned}
\end{equation}
Update step:
\begin{equation}\label{updateClassical}
\begin{aligned}
&X_{t_k}^{(i),a} = X_{t_k}^{(i),f} + K_k \left( \Delta Y_k + C^{\frac{1}{2}}\tilde{V}_k^{(i)} - hg\left(X_{t_k}^{(i),f}\right)\right),\\
&X_{t_k}^{(i),a} - \bar{x}_k^{a} = X_{t_k}^{(i),f}-\bar{x}_k^{f} + K_k\left(C^{\frac{1}{2}}\left(\tilde{V}_k^{(i)}-\bar{v}_k\right) - h\left(g\left(X_{t_k}^{(i),f}\right)-\bar{g}_k^{f}\right)\right)
\end{aligned}
\end{equation}
\end{algo}

\begin{remark}
In the case of a linear observation operator, i.e. $g(x) = Gx$, an alternative interpretation of the Kalman gain matrix is obtained from the following optimization problem (see Proposition 3.2 in \cite{kelly2014}): minimizing the mean-square distance of the ensemble members to the forecast as well as the information provided by the innovation, i.e. minimizing the functional
\begin{equation}\label{cost}
\mathcal{J}_k^{(i)}(X) := \frac{1}{2} \left\| \Delta Y_k^{(i)} - hGX\right\|_{C}^2 + \frac{1}{2} \left\| X - X_{t_k}^{(i),f}\right\|_{P_k^f}^2
\end{equation}
where for a positive definite matrix $A$ we use the Mahalanobis norm
\begin{equation}
\| x \|_A^2 := \langle A^{-1} x, x \rangle = x^T A^{-1} x,
\end{equation}
leads to the minimizer given by (\ref{classicalUpdate}) for each $i=1,...,M$ with gain
\begin{equation}
K_k = P_k^{f}G^T \left(C+hGP_k^{f}G^T\right)^{-1}, \hspace{1cm} P_k^{f} = \frac{1}{M-1} E_k^{f}\left(E_k^{f}\right)^T.
\end{equation}

\smallskip
\noindent
In the general nonlinear case, however, we cannot employ such a derivation since the above minimization problem cannot be solved analytically for general $g$. Thus we choose a similar structure of the gain matrix to obtain Equation (\ref{KalmanGain}) by replacing
\begin{equation}
\begin{aligned}
P_k^{f}G^T &\leadsto \frac{1}{M-1}E_k^{f}\left(\mathcal{G}_k^{f}\right)^T,\\
GP_k^{f}G^T &\leadsto \frac{1}{M-1}\mathcal{G}_k^{f}\left(\mathcal{G}_k^{f}\right)^T.
\end{aligned}
\end{equation}
\end{remark}

\subsection{The modified formulation}\label{modifiedAlgo}
In discrete time, the choice (\ref{noiseReplacement}) translates into the addition of similarly motivated terms of the form
\begin{equation}
\begin{aligned}
Q^{\frac{1}{2}}\tilde{W}_k^{(i)} &:= \frac{h}{2} Q \left(P_{k-1}^{a}\right)^{-1} \left(X_{t_{k-1}}^{(i),a}-\bar{x}_{k-1}^{a}\right),\\
C^{\frac{1}{2}}\tilde{V}_k^{(i)} &:= \frac{h}{2}\left(g\left(X_{t_k}^{(i),f}\right)-\bar{g}_k^{f}\right).
\end{aligned}
\end{equation}
This in total yields a discrete-time algorithm for the ensemble members as well as the centered ensemble members:
\begin{algo}\label{ModifiedAlgorithm}
Forecast step:
\begin{equation}\label{centerModF}
\begin{aligned}
&X_{t_k}^{(i),f} = X_{t_{k-1}}^{(i),a} + hf\left(X_{t_{k-1}}^{(i),a}\right) + \frac{h}{2} Q \left(P_{k-1}^{a}\right)^{-1} \left(X_{t_{k-1}}^{(i),a}-\bar{x}_{k-1}^{a}\right),\\
&X_{t_k}^{(i),f} - \bar{x}_k^{f} = X_{t_{k-1}}^{(i),a}-\bar{x}_{k-1}^{a} + h\left(f\left(X_{t_{k-1}}^{(i),a}\right) - \bar{f}_{k-1}^{a}\right) + \frac{h}{2}Q\left(P_{k-1}^{a}\right)^{-1}\left(X_{t_{k-1}}^{(i),a}-\bar{x}_{k-1}^{a}\right)
\end{aligned}
\end{equation}
Update step:
\begin{equation}\label{centerModA}
\begin{aligned}
&X_{t_k}^{(i),a} = X_{t_k}^{(i),f} + K_k \left( \Delta Y_k - \frac{h}{2}\left(g\left(X_{t_k}^{(i),f}\right) + \bar{g}_k^f\right)\right),\\
&X_{t_k}^{(i),a}-\bar{x}_k^{a} = X_{t_k}^{(i),f} - \bar{x}_k^{f} - \frac{h}{2}K_k\left(g\left(X_{t_k}^{(i),f}\right)-\bar{g}_k^{f}\right)
\end{aligned}
\end{equation}
\end{algo}

\section{The continuous time limit}\label{Results}

\noindent
Similar to \cite{bloemker2018}, introduce for $t \in [t_{k-1}, t_k]$ the notation
\begin{equation}
\eta(t) = t_{k-1}, \hspace{0.5cm} \eta_{+}(t) = t_k, \hspace{0.5cm} \nu(t) = k-1, \hspace{0.5cm} \nu_{+}(t) = k.
\end{equation}
For the classical formulation we obtain the following continuous time limit result:
\begin{theorem}\label{ClassicalLimit}
Consider Algorithm \ref{ClassicalAlgo}. If the initial ensemble satisfies the following properties
\begin{align}
&\sum_{i=1}^M \mathds{E}\left[\left\|X_0^{(i),a} - X_0^{(i)}\right\|^2\right] \in O\left(h^{2\gamma}\right),\label{initialProp1}\\
&\sum_{i=1}^M \mathds{E}\left[\left\|X_0^{(i),a} - \bar{x}_0^{a}\right\|^2\right] <\infty,\hspace{0.5cm} \mathds{E}\left[\left(\sum_{i=1}^M \left\|X_0^{(i),a}-\bar{x}_0^{a}\right\|^2\right)^2\right] < \infty\label{initialProp2}
\end{align}
where $\gamma < \frac{1}{2}$ is the Hölder-coefficient coming from $W$, $V$, and $Y$, then there exists a continuous process $(L(s))_{s \geq 0}$ such that
\begin{equation}
\sup_{t \in [0,T]}\mathds{E}\left[\text{e}^{-\int_0^t2L(s){\rm d}s}\left(\sum_{i=1}^M \left\|X_{\eta(t)}^{(i),a}-X_t^{(i)}\right\|^2\right)\right] \in O\left(h^{2\gamma}\right).
\end{equation}
\end{theorem}
\noindent
Assuming further an integrability assumption on $L$ we also obtain an explicit convergence rate:
\begin{corollary}\label{ImprovedRates}
If, moreover, there exists a $\delta > 0$ such that
\begin{equation}\label{boundedL}
\sup_{t\in [0,T]}\mathds{E}\left[\text{e}^{2\delta \int_0^t L(s){\rm d}s}\right] <\infty,
\end{equation}
then
\begin{equation}
\sup_{t\in [0,T]}\mathds{E}\left[\left(\sum_{i=1}^M \left\| X_{\eta(t)}^{(i),a}-X_t^{(i)}\right\|^2\right)^{\frac{\delta}{1+\delta}}\right] \in O\left(h^{2\gamma\frac{ \delta}{1+\delta}}\right).
\end{equation}
\end{corollary}

\medskip
\noindent
To establish the exponential moment estimate (\ref{boundedL}) turns out to be rather difficult. The reason for that lies in the following fact: $L(t)$ is an affine linear functional of
\begin{equation}\label{Vt}
\mathcal{V}_t := \frac{1}{M-1} \sum_{i=1}^M \left\|X_t^{(i)} - \bar{x}_t\right\|^2
\end{equation}
which reduces the problem of finding exponential moment estimates of $\int_0^tL(s){\rm d}s$ to the same problem for $\int_0^t\mathcal{V}_s{\rm d}s$. However, $\mathcal{V}$ satisfies the following differential inequality (see Lemma \ref{boundVtclassical})
\begin{equation}
{\rm d} \mathcal{V}_t \leq 2(Lf)_{+} \mathcal{V}_t {\rm d}t + \text{tr}(Q) {\rm d}t + {\rm d}N_t
\end{equation}
where
\begin{equation}\label{dissipativityf}
(Lf)_{+} := \sup_{x \neq y} \frac{\left \langle f(x) -f(y), x -y \right \rangle}{\|x-y\|^2}
\end{equation}
and
\begin{equation}
{\rm d}N_t =  \frac{2}{M-1} \sum_{i=1}^M \left \langle X_t^{(i)} - \bar{x}_t, Q^{\frac{1}{2}}{\rm d}W_t^{(i)} + \frac{1}{M-1}E_t\mathcal{G}_t^TC^{-\frac{1}{2}}{\rm d}V_t^{(i)}\right \rangle
\end{equation}
hence the above problem requires the control of the stochastic integral with respect to $N$. But note that the quadratic variation of $N$ satisfies
\begin{equation}\label{VquadVar2}
\frac{{\rm d}}{{\rm d}t} \langle N\rangle_t \leq \frac{4t}{(M-1)^2}\left(\left\|Q^{\frac{1}{2}}\right\|^2 + \frac{4M\|g\|_{\infty}^2\left\|C^{-\frac{1}{2}}\right\|^2}{M-1}\mathcal{V}_t\right)\mathcal{V}_t
\end{equation}
which gives no effective control on the growth of $N$ and hence on $\mathcal{V}$.

\medskip

\begin{remark}
Whithout the use of a stochastic perturbation of the observation in the classical EnKF formulation (\ref{classical}), however, the strategy of the proof of Lemma 21 in \cite{deWiljes2017} is indeed applicable to our setting hence we can find a $\delta > 0$ such that for the corresponding process $L$ it holds
\begin{equation}
\sup_{t\in [0,T]}\mathds{E}\left[\text{e}^{2\delta \int_0^t L(s){\rm d}s}\right] <\infty.
\end{equation}
\end{remark}

\medskip
\noindent
As already mentioned in the Introduction we obtain better error estimates for the modified formulation:
\begin{theorem}\label{ModifiedLimit}
Consider Algorithm \ref{ModifiedAlgorithm}. If the initial ensemble satisfies
\begin{equation}
\sum_{i=1}^M\mathds{E}\left[\left\|X_0^{(i),a}-X_0^{(i)}\right\|^2\right] \in O\left(h\right)
\end{equation}
and has bounded inverse covariance matrix, then
\begin{equation}
\mathds{E}\left[\sup_{t\in [0,T]}\sum_{i=1}^M \left\|X_{\eta(t)}^{(i),a}-X_t^{(i)}\right\|^2\right] \in O\left(h\right).
\end{equation}
\end{theorem}
\noindent
Note that this statement is stronger than the result in Theorem \ref{ClassicalLimit} which is due to the fact that the coefficients of (\ref{modified}) are bounded as will be shown in Section \ref{Proofs}. In particular, this yields that for all $t \in [0,T]$
\begin{equation}
\sum_{i=1}^M \left\|X_{\eta(t)}^{(i)}-X_t^{(i)}\right\| \leq \tilde{C}|t-\eta(t)|^{\rho}
\end{equation}
for some $\rho <\frac{1}{2}$, thus we may concentrate our analysis on the time points $t_k$ of the time interval partition.

\section{Proofs of Theorem \ref{ClassicalLimit} and Theorem \ref{ModifiedLimit}}\label{Proofs}

\subsection{Preliminaries}

\noindent
For the classical formulation we introduce the 'continuous embedding' of the update step
\begin{equation}
\begin{aligned}
\widehat{X}_t^{(i)} &:= X_{\eta(t)}^{(i),a} + \int_{\eta(t)}^t f\left(X_{\eta(s)}^{(i),a}\right) {\rm d}s + \int_{\eta(t)}^tQ^{\frac{1}{2}}{\rm d}W_s^{(i)}\\
&\hspace{0.5cm}+ \int_{\eta(t)}^t K_{\nu_{+}(s)} {\rm d}Y_s + \int_{\eta(t)}^t K_{\nu_{+}(s)} C^{\frac{1}{2}} {\rm d}V_s^{(i)} - \int_{\eta(t)}^t K_{\nu_{+}(s)} g\left(X_{\eta_+(s)}^{(i),f}\right) {\rm d}s
\end{aligned}
\end{equation}
and consider the decomposition
\begin{equation}
\left\|X_{\eta(t)}^{(i),a}-X_t^{(i)}\right\|^2 \leq 2 \left\|X_{\eta(t)}^{(i),a}-\widehat{X}_t^{(i)}\right\|^2 + 2 \left\|\widehat{X}_t^{(i)}-X_t^{(i)}\right\|^2
\end{equation}
where
\begin{equation}\label{classicalSemimartingale}
\begin{aligned}
\widehat{X}_t^{(i)} - X_t^{(i)}&= X_{\eta(t)}^{(i),a} - X_{\eta(t)}^{(i)}\\
&\hspace{0.5cm} + \int_{\eta(t)}^t f\left(X_{\eta(s)}^{(i),a}\right)-f\left(\widehat{X}_s^{(i)}\right){\rm d}s + \int_{\eta(t)}^t f\left(\widehat{X}_s^{(i)}\right)-f\left(X_s^{(i)}\right){\rm d}s\\
&\hspace{0.5cm} + \int_{\eta(t)}^t \left(K_{\nu_{+}(s)} - \frac{1}{M-1}E_s\mathcal{G}_s^TC^{-1}\right){\rm d}Y_s\\
&\hspace{0.5cm} + \int_{\eta(t)}^t \left(K_{\nu_{+}(s)}-\frac{1}{M-1}E_s\mathcal{G}_s^TC^{-1}\right)C^{\frac{1}{2}}{\rm d}V_s^{(i)}\\
&\hspace{0.5cm} - \int_{\eta(t)}^t \left(K_{\nu_{+}(s)}g\left(X_{\eta_+(s)}^{(i),f}\right)-\frac{1}{M-1}E_s\mathcal{G}_s^TC^{-1}g\left(\widehat{X}_s^{(i)}\right)\right){\rm d}s\\
&\hspace{0.5cm} - \int_{\eta(t)}^t \frac{1}{M-1}E_s\mathcal{G}_s^TC^{-1}\left(g\left(\widehat{X}_s^{(i)}\right)-g\left(X_s^{(i)}\right)\right){\rm d}s.
\end{aligned}
\end{equation}

\medskip
\noindent
In the modified formulation consider the decomposition
\begin{equation}
\left\|X_{\eta(t)}^{(i),a}-X_t^{(i)}\right\|^2 \leq 2 \left\|X_{\eta(t)}^{(i),a}-X_{\eta(t)}^{(i)}\right\|^2 + 2 \left\|X_{\eta(t)}^{(i)}-X_t^{(i)}\right\|^2
\end{equation}
where
\begin{equation}\label{modifiedSemimartingale}
\begin{aligned}
&X_{\eta(t)}^{(i),a}- X_{\eta(t)}^{(i)}\\
&= X_0^{(i),a} - X_0^{(i)}\\
&\hspace{0.5cm}+\int_0^{\eta(t)} f\left(X_{\eta(s)}^{(i),a}\right)-f\left(X_s^{(i)}\right){\rm d}s\\
&\hspace{0.5cm} + \int_0^{\eta(t)} \frac{1}{2}Q\left(\left(P_{\nu(s)}^{a}\right)^{-1}\left(X_{\eta(s)}^{(i),a}-\bar{x}_{\nu(s)}^{a}\right)-P_s^{-1}\left(X_s^{(i)}-\bar{x}_s\right)\right){\rm d}s\\
&\hspace{0.5cm} + \int_0^{\eta(t)}\left(K_{\nu_+(s)}-\frac{1}{M-1}E_s\mathcal{G}_s^TC^{-1}\right){\rm d}Y_s\\
&\hspace{0.5cm} - \int_0^{\eta(t)}\frac{1}{2}\left(K_{\nu_+(s)}\left(g\left(X_{\eta_+(s)}^{(i),f}\right)+\bar{g}_{\nu_+(s)}^{f}\right)-\frac{1}{M-1}E_s\mathcal{G}_s^TC^{-1}\left(g\left(X_s^{(i)}\right)+\bar{g}_s\right)\right){\rm d}s.
\end{aligned}
\end{equation}

\medskip
\noindent
Further note that
\begin{equation}\label{Decomp}
\begin{aligned}
&K_{\nu_{+}(s)}g(x)-\frac{1}{M-1}E_s\mathcal{G}_s^TC^{-1}g(y)\\
&= K_{\nu_{+}(s)}\left(g(x)-g(y)\right) + \left(K_{\nu_{+}(s)}-\frac{1}{M-1}E_s\mathcal{G}_s^TC^{-1}\right)g(y)\\
&= \left(K_{\nu_{+}(s)}-\frac{1}{M-1}E_s\mathcal{G}_s^TC^{-1}\right)g(x)+\frac{1}{M-1}E_s\mathcal{G}_s^TC^{-1}\left(g(x)-g(y)\right).
\end{aligned}
\end{equation}

\medskip
\noindent
An essential tool in the analysis is the uniform control of the Lipschitz constants of the coefficients. The coefficients that both the classical and the modified formulation have in common are
\begin{align}
&\left\|K_{\nu_{+}(t)} - \frac{1}{M-1}E_t\mathcal{G}_t^TC^{-1}\right\|^2,\label{coeff1}\\
&\left\|K_{\nu_{+}(t)}\right\|^2,\label{coeff2}\\
&\left\| \frac{1}{M-1}E_t\mathcal{G}_t^TC^{-1}\right\|^2,\label{coeff3}
\end{align}
thus we provide the following preliminary estimates:

\begin{lemma}\label{coundCommonCoeff}
There exist constants $C_1, C_2, C_3$ independent of $h$ such that for both Algorithms \ref{ClassicalAlgo} and \ref{ModifiedAlgorithm} it holds
\begin{equation}\label{first}
\begin{aligned}
\left\|K_{\nu_{+}(t)} - \frac{1}{M-1}E_t\mathcal{G}_t^TC^{-1}\right\|^2&\leq C_1\left(h^2\left(\mathcal{V}_{\nu_{+}(t)}^{f}\right)^2\right.\\
&\hspace{1.5cm}\left. + \left(1+\mathcal{V}_t\right)\left(\frac{1}{M-1}\sum_{i=1}^M \left\|X_{\eta_{+}(t)}^{(i),f}-X_t^{(i)}\right\|^2\right)\right),
\end{aligned}
\end{equation}
\begin{align}
&\left\|K_{\nu_{+}(t)}\right\|^2\leq C_2 \mathcal{V}_{\nu_{+}(t)}^{f},\label{second}\\
&\left\| \frac{1}{M-1}E_t\mathcal{G}_t^TC^{-1}\right\|^2 \leq C_3 \mathcal{V}_t\label{third}
\end{align}
where
\begin{equation}
\begin{aligned}
\mathcal{V}_{\nu_{+}(t)}^{f} &:= \frac{1}{M-1} \sum_{i=1}^M \left\|X_{\eta_{+}(t)}^{(i),f} - \bar{x}_{\nu_{+}(t)}^{f}\right\|^2\\
\mathcal{V}_t &:= \frac{1}{M-1} \sum_{i=1}^M \left\|X_t^{(i)} - \bar{x}_t\right\|^2.
\end{aligned}
\end{equation}
\end{lemma}

\medskip
\noindent
In both the classical and the modified formulation, we can control $\mathcal{V}_t$ as well as $\mathcal{V}_{\nu_{+}(t)}^{f}$ in a sense made precise in Appendix \ref{appendixVt} and Appendix \ref{appendixVk}, respectively.

\begin{proof}
On (\ref{first}):
\begin{equation}
\begin{aligned}
\text{(\ref{coeff1})}&= \left\|\frac{1}{M-1}E_{\nu_{+}(t)}^{f}\left(\mathcal{G}_{\nu_{+}(t)}^{f}\right)^T\left( \left(C+\frac{h}{M-1}\mathcal{G}_{\nu_{+}(t)}^{f}\left(\mathcal{G}_{\nu_{+}(t)}^{f}\right)^T\right)^{-1} - C^{-1}\right) \right.\\
&\hspace{1cm}\left. + \frac{1}{M-1} \left(E_{\nu_{+}(t)}^{f} \left(\mathcal{G}_{\nu_{+}(t)}^{f}\right)^T - E_t \mathcal{G}_t^T\right) C^{-1}\right\|^2\\
&\leq 2\left\|\frac{1}{M-1}E_{\nu_{+}(t)}^{f}\left(\mathcal{G}_{\nu_{+}(t)}^{f}\right)^T\right\|^2\left\| \left(C+\frac{h}{M-1}\mathcal{G}_{\nu_{+}(t)}^{f}\left(\mathcal{G}_{\nu_{+}(t)}^{f}\right)^T\right)^{-1} - C^{-1}\right\|^2 \\
&\hspace{0.5cm} + 2\left\|\frac{1}{M-1} \left(E_{\nu_{+}(t)}^{f} \left(\mathcal{G}_{\nu_{+}(t)}^{f}\right)^T - E_t \mathcal{G}_t^T\right) \right\|^2 \left\|C^{-1}\right\|^2.
\end{aligned}
\end{equation}

\medskip
\noindent
For this we estimate by boundedness of $g$
\begin{equation}\label{EGf}
\begin{aligned}
\left\|\frac{1}{M-1}E_{\nu_{+}(t)}^{f}\left(\mathcal{G}_{\nu_{+}(t)}^{f}\right)^T\right\| &= \left\|\frac{1}{M-1}\sum_{i=1}^M \left(X_{\eta_{+}(t)}^{(i),f} - \bar{x}_{\nu_{+}(t)}^{f}\right)\left(g\left(X_{\eta_{+}(t)}^{(i),f}\right)-\bar{g}_{\nu_{+}(t)}^{f}\right)^T\right\|\\
&\leq 2\|g\|_{\infty}\sqrt{\frac{M}{M-1}}\left(\mathcal{V}_{\nu_{+}(t)}^{f}\right)^{\frac{1}{2}}.
\end{aligned}
\end{equation}

\medskip
\noindent
Further
\begin{equation}
\begin{aligned}
&\left\| \left(C+\frac{h}{M-1}\mathcal{G}_{\nu_{+}(t)}^{f}\left(\mathcal{G}_{\nu_{+}(t)}^{f}\right)^T\right)^{-1} - C^{-1}\right\| \\
&= \left\|\left(C+\frac{h}{M-1}\mathcal{G}_{\nu_{+}(t)}^{f}\left(\mathcal{G}_{\nu_{+}(t)}^{f}\right)^T\right)^{-1} \left( {\rm Id} - \left(C+\frac{h}{M-1}\mathcal{G}_{\nu_{+}(t)}^{f}\left(\mathcal{G}_{\nu_{+}(t)}^{f}\right)^T\right)C^{-1}\right)\right\|\\
&= \left\|\left(C+\frac{h}{M-1}\mathcal{G}_{\nu_{+}(t)}^{f}\left(\mathcal{G}_{\nu_{+}(t)}^{f}\right)^T\right)^{-1} \left(-\frac{h}{M-1}\mathcal{G}_{\nu_{+}(t)}^{f}\left(\mathcal{G}_{\nu_{+}(t)}^{f}\right)^TC^{-1}\right)\right\|\\
&\leq  h \left\|C^{-1}\right\|^2 \left\|\frac{1}{M-1}\mathcal{G}_{\nu_{+}(t)}^{f}\left(\mathcal{G}_{\nu_{+}(t)}^{f}\right)^T\right\|\leq h \frac{4\|g\|_{\infty}^2\left\|C^{-1}\right\|^2M}{M-1}.
\end{aligned}
\end{equation}

\medskip
\noindent
Finally
\begin{equation}
\begin{aligned}
&\left\|\frac{1}{M-1} \left(E_{\nu_{+}(t)}^{f} \left(\mathcal{G}_{\nu_{+}(t)}^{f}\right)^T - E_t \mathcal{G}_t^T\right) \right\|\\
&= \left\| \frac{1}{M-1} \sum_{i=1}^M \left(\left(X_{\eta_{+}(t)}^{(i),f} - \bar{x}_{\nu_{+}(t)}^{f}\right) - \left(X_t^{(i)}-\bar{x}_t\right)\right)\left(g\left(X_{\eta_{+}(t)}^{(i),f}\right) - g\left(\bar{x}_{\nu_{+}(t)}^{f}\right)\right)^T\right.\\
&\hspace{3cm}\left. + \left(X_t^{(i)}-\bar{x}_t\right)\left(\left(g\left(X_{\eta_{+}(t)}^{(i),f}\right)-g\left(\bar{x}_{\nu_{+}(t)}^{f}\right)\right) - \left(g\left(X_t^{(i)}\right) - g\left(\bar{x}_t\right)\right)\right)^T\right\|\\
&\leq \left(\frac{1}{M-1}\sum_{i=1}^M \left\|\left(X_{\eta_{+}(t)}^{(i),f} - \bar{x}_{\nu_{+}(t)}^{f}\right) - \left(X_t^{(i)}-\bar{x}_t\right)\right\|^2\right)^{\frac{1}{2}}\left(\frac{1}{M-1}\sum_{i=1}^M \left\|g\left(X_{\eta_{+}(t)}^{(i),f}\right) - g\left(\bar{x}_{\nu_{+}(t)}^{f}\right)\right\|^2\right)^{\frac{1}{2}}\\
&\hspace{0.5cm} + \left(\frac{1}{M-1}\sum_{i=1}^M \left\|X_t^{(i)}-\bar{x}_t\right\|^2\right)^{\frac{1}{2}}\\
&\hspace{2cm}\left(\frac{1}{M-1}\sum_{i=1}^M \left\|\left(g\left(X_{\eta_{+}(t)}^{(i),f}\right)-g\left(\bar{x}_{\nu_{+}(t)}^{f}\right)\right) - \left(g\left(X_t^{(i)}\right) - g\left(\bar{x}_t\right)\right)\right\|^2\right)^{\frac{1}{2}}\\
&\leq 2\left(2\|g\|_{\infty}\sqrt{\frac{M}{M-1}} + \|g\|_{\text{Lip}}\left(\mathcal{V}_t\right)^{\frac{1}{2}}\right)\left(\frac{1}{M-1} \sum_{i=1}^M \left\|X_{\eta_{+}(t)}^{(i),f} - X_t^{(i)}\right\|^2\right)^{\frac{1}{2}}.
\end{aligned}
\end{equation}

\medskip
\noindent
In total we obtain
\begin{equation}\label{estimateKkclassical}
\begin{aligned}
&\left\|K_{\nu_{+}(t)} - \frac{1}{M-1}E_t\mathcal{G}_t^TC^{-1}\right\|^2\\
&\leq 16\left\|C^{-1}\right\|^2\left(h^2\frac{8\|g\|_{\infty}^4\|g\|_{\text{Lip}}^2\left\|C^{-1}\right\|^2M^2}{(M-1)^2}\left(\mathcal{V}_{\nu_{+}(t)}^{f}\right)^2 \right.\\
&\hspace{2.5cm}\left.+ \left(\frac{4\|g\|_{\infty}^2M}{M-1} + \|g\|_{\text{Lip}}^2\mathcal{V}_t\right)\left(\frac{1}{M-1} \sum_{i=1}^M \left\|X_{\eta_{+}(t)}^{(i),f} - X_t^{(i)}\right\|^2\right)\right).
\end{aligned}
\end{equation}

\medskip
\noindent
On (\ref{second}): using (\ref{EGf}) we obtain
\begin{equation}\label{estimateKk}
\begin{aligned}
\left\|K_{\nu_{+}(t)}\right\| &\leq \left\|\frac{1}{M-1}E_{\nu_{+}(t)}^{f}\left(\mathcal{G}_{\nu_{+}(t)}^{f}\right)^T\right\|\left\|C^{-1}\right\|\\
&\leq 2\|g\|_{\infty}\left\|C^{-1}\right\|\sqrt{\frac{M}{M-1}} \left(\mathcal{V}_{\nu_{+}(t)}^{f}\right)^{\frac{1}{2}},
\end{aligned}
\end{equation}
i.e.
\begin{equation}\label{K2}
\left\|K_{\nu_{+}(t)}\right\|^2 \leq \frac{4\|g\|_{\infty}^2\left\|C^{-1}\right\|^2M}{M-1}\mathcal{V}_{\nu_{+}(t)}^{f}.
\end{equation}

\medskip
\noindent
On (\ref{third}): similar to (\ref{EGf}) we obtain
\begin{equation}\label{Kcont}
\left\| \frac{1}{M-1}E_t\mathcal{G}_t^TC^{-1}\right\|^2 \leq \frac{4\|g\|_{\infty}^2\left\|C^{-1}\right\|^2M}{M-1}\mathcal{V}_t.
\end{equation}
\end{proof}

\subsection{Proof of Theorem \ref{ClassicalLimit}}

\noindent
Note that pathwise it holds by H\"older-continuity of $W^{(i)}$, $V^{(i)}$, and $Y$
\begin{equation}
\begin{aligned}
\left\|X_{\eta(t)}^{(i),a} - \widehat{X}_t^{(i)}\right\| &\leq h \left(\|f\|_{\infty} + \|g\|_{\infty}\left\|K_{\nu_{+}(t)}\right\|\right)\\
&\hspace{0.5cm} + h^{\gamma}\left(\left\|Q^{\frac{1}{2}}\right\| \left\|W^{(i)}\right\|_{\text{H\"ol}} + \|Y\|_{\text{H\"ol}}\left\|K_{\nu_{+}(t)}\right\| + \left\|C^{\frac{1}{2}}\right\|\left\|V^{(i)}\right\|_{\text{H\"ol}}\left\|K_{\nu_{+}(t)}\right\|\right)
\end{aligned}
\end{equation}
where $\gamma <\frac{1}{2}$, and thus
\begin{equation}\label{firstest}
\begin{aligned}
\left\|X_{\eta(t)}^{(i),a} - \widehat{X}_t^{(i)}\right\|^2 &\leq 5\left( h^2 \left(\|f\|_{\infty}^2 + \|g\|_{\infty}^2\left\|K_{\nu_{+}(t)}\right\|^2\right)\right.\\
&\hspace{1cm} + h^{2\gamma}\left(\left\|Q^{\frac{1}{2}}\right\|^2 \left\|W^{(i)}\right\|_{\text{H\"ol}}^2 + \|Y\|_{\text{H\"ol}}^2\left\|K_{\nu_{+}(t)}\right\|^2 \right.\\
&\hspace{5cm}\left.\left.+ \left\|C^{\frac{1}{2}}\right\|^2\left\|V^{(i)}\right\|_{\text{H\"ol}}^2\left\|K_{\nu_{+}(t)}\right\|^2\right)\right).
\end{aligned}
\end{equation}
By Lemma \ref{coundCommonCoeff} and Appendix \ref{appendixVk}, $\left\|K_k\right\|^2$ is bounded in expectation for any $k$ and we get
\begin{equation}\label{analyzedContembClassical}
\mathds{E}\left[\left\|X_{\eta(t)}^{(i),a} - \widehat{X}_t^{(i)}\right\|^2\right] \in O\left(h^{2\gamma}\right).
\end{equation}

\medskip
\noindent
Further observe that the observation process $Y$ can be written as
\begin{equation}
{\rm d}Y_t = g\left(X_t^{\text{ref}}\right){\rm d}t + C^{\frac{1}{2}}{\rm d}V_t
\end{equation}
where $X^{\text{ref}}$ is the reference trajectory that generates the observations. Using this in (\ref{classicalSemimartingale}) we obtain by the It\^{o} formula
\begin{equation}
\begin{aligned}
&\frac{1}{2}{\rm d}\left\|\widehat{X}_t^{(i)}-X_t^{(i)}\right\|^2\\
&= \left \langle \widehat{X}_t^{(i)} - X_t^{(i)}, f\left(\widehat{X}_t^{(i)}\right) - f\left(X_t^{(i)}\right) \right \rangle {\rm d}t\\
&\hspace{0.5cm} - \left \langle \widehat{X}_t^{(i)} - X_t^{(i)}, \frac{1}{M-1} E_t \mathcal{G}_t^T C^{-1}\left(g\left(\widehat{X}_t^{(i)}\right)-g\left(X_t^{(i)}\right)\right)\right \rangle {\rm d}t\\
&\hspace{0.5cm} + \left \langle \widehat{X}_t^{(i)} - X_t^{(i)}, f\left(X_{\eta(t)}^{(i),a}\right) - f\left(\widehat{X}_t^{(i)}\right) + \left(K_{\nu_{+}(t)} - \frac{1}{M-1}E_t\mathcal{G}_t^TC^{-1}\right)g\left(X_t^{\text{ref}}\right) \right.\\
&\hspace{3.5cm}\left.- \left(K_{\nu_{+}(t)}g\left(X_{\eta_{+}(t)}^{(i),f}\right) - \frac{1}{M-1}E_t\mathcal{G}_t^TC^{-1}g\left(\widehat{X}_t^{(i)}\right)\right)\right \rangle {\rm d}t\\
&\hspace{0.5cm} + \text{tr}\left( \left(K_{\nu_{+}(t)} - \frac{1}{M-1}E_t\mathcal{G}_t^TC^{-1}\right)C\left(K_{\nu_{+}(t)} - \frac{1}{M-1} E_t\mathcal{G}_t^TC^{-1}\right)^T\right){\rm d}t\\
&\hspace{0.5cm} + \left \langle \widehat{X}_t^{(i)} - X_t^{(i)}, \left(K_{\nu_{+}(t)}-\frac{1}{M-1}E_t\mathcal{G}_t^TC^{-1}\right)C^{\frac{1}{2}}\left({\rm d}V_t + {\rm d}V_t^{(i)}\right)\right \rangle\\
&\leq \left(\frac{1}{2} + \|f\|_{\text{Lip}} + \|g\|_{\text{Lip}}^2\left\|C^{-1}\right\|\mathcal{V}_t\right)\left\|\widehat{X}_t^{(i)}-X_t^{(i)}\right\|^2{\rm d}t\\
&\hspace{0.5cm}+\left( 2\|f\|_{\text{Lip}}^2\left\|X_{\eta(t)}^{(i),a}-\widehat{X}_t^{(i)}\right\|^2 + 2\|g\|_{\text{Lip}}^2\left\|K_{\nu_{+}(t)}\right\|^2\left\|X_{\eta_{+}(t)}^{(i),f}-\widehat{X}_t^{(i)}\right\|^2 \right.\\
&\hspace{1.5cm}\left.+ \left(4\|g\|_{\infty}^2+\|C\|\right)\left\|K_{\nu_{+}(t)}-\frac{1}{M-1}E_t\mathcal{G}_t^TC^{-1}\right\|_{F}^2\right){\rm d}t\\
&\hspace{0.5cm} + \left \langle \widehat{X}_t^{(i)} - X_t^{(i)}, \left(K_{\nu_{+}(t)}-\frac{1}{M-1}E_t\mathcal{G}_t^TC^{-1}\right)C^{\frac{1}{2}}\left({\rm d}V_t + {\rm d}V_t^{(i)}\right)\right \rangle.
\end{aligned}
\end{equation}

\medskip
\noindent
Observe that
\begin{equation}
\begin{aligned}
&\sum_{i=1}^M \left\|X_{\eta_{+}(t)}^{(i),f} - \widehat{X}_t^{(i)}\right\|^2\\
&= \sum_{i=1}^M \left\|X_{\eta(t)}^{(i),a} - \widehat{X}_t^{(i)} + hf\left(X_{\eta(t)}^{(i),a}\right) + Q^{\frac{1}{2}}\tilde{W}_{\nu_{+}(t)}^{(i)}\right\|^2\\
&\leq 3\left(\sum_{i=1}^M \left\|X_{\eta(t)}^{(i),a} - \widehat{X}_t^{(i)}\right\|^2 + h^2\|f\|_{\infty}^2M + \left\|Q^{\frac{1}{2}}\right\|^2 \sum_{i=1}^M \left\|\tilde{W}_{\nu_{+}(t)}^{(i)}\right\|^2\right)
\end{aligned}
\end{equation}
by boundedness of $f$. With this and Lemma \ref{coundCommonCoeff} we obtain in total
\begin{equation}
\frac{1}{2}{\rm d}\sum_{i=1}^M \left\|\widehat{X}_t^{(i)} - X_t^{(i)}\right\|^2 \leq L(t)\sum_{i=1}^M \left\|\widehat{X}_t^{(i)}-X_t^{(i)}\right\|^2 {\rm d}t + R(t){\rm d}t + {\rm d}\mathcal{N}_t
\end{equation}
\noindent
where
\begin{equation}
\begin{aligned}
L(t) &:= \frac{1}{2}+\|f\|_{\text{Lip}} + \|g\|_{\text{Lip}}^2\|C^{-1}\|\mathcal{V}_t  + \kappa \left(\frac{4\|g\|_{\infty}^2M}{M-1} + \|g\|_{\text{Lip}}^2\mathcal{V}_t\right),\\
R(t) &:= \left(2\|f\|_{\text{Lip}}^2 + 6\|g\|_{\text{Lip}}^2\left\|K_{\nu_{+}(t)}\right\|^2+ \kappa\left(\frac{4\|g\|_{\infty}^2M}{M-1} + \|g\|_{\text{Lip}}^2\mathcal{V}_t\right)\right)\sum_{i=1}^M \left\|X_{\eta(t)}^{(i),a} - \widehat{X}_t^{(i)}\right\|^2\\
&\hspace{0.5cm} + 6h^2\|f\|_{\infty}^2M + 6\left\|Q^{\frac{1}{2}}\right\|^2\sum_{i=1}^M \left\|\tilde{W}_{\nu_{+}(t)}^{(i)}\right\|^2\\
&\hspace{0.5cm} + \kappa\left(h^2\frac{2\|g\|_{\infty}^4\|g\|_{\text{Lip}}^2\left\|C^{-1}\right\|^2M}{M-1}\left(\mathcal{V}_{\nu_{+}(t)}^{f}\right)^2\right.\\
&\hspace{2cm}\left. + \left(\frac{4\|g\|_{\infty}^2M}{M-1} + \|g\|_{\text{Lip}}^2\mathcal{V}_t\right)\left(h^2\|f\|_{\infty}^2M+\left\|Q^{\frac{1}{2}}\right\|^2\sum_{i=1}^M\left\|\tilde{W}_{\nu_{+}(t)}^{(i)}\right\|^2\right)\right),\\
{\rm d}\mathcal{N}_t &:= \sum_{i=1}^M \left \langle \widehat{X}_t^{(i)} - X_t^{(i)}, \left(K_{\nu_{+}(t)}-\frac{1}{M-1}E_t\mathcal{G}_t^TC^{-1}\right)C^{\frac{1}{2}}\left({\rm d}V_t + {\rm d}V_t^{(i)}\right)\right \rangle
\end{aligned}
\end{equation}
\medskip
\noindent
with $\kappa := 64\|C^{-1}\|^2\left(4\|g\|_{\infty}^2+\|C\|\right)\frac{M}{M-1}$. By the It\^{o} product rule
\begin{equation}
\begin{aligned}
&{\rm d} \text{e}^{-\int_0^t2L(s){\rm d}s}\left(\sum_{i=1}^M \left\|\widehat{X}_t^{(i)}-X_t^{(i)}\right\|^2\right) \leq 2\text{e}^{-\int_0^t2L(s){\rm d}s}\left(R(t){\rm d}t + {\rm d}\mathcal{N}_t\right)\\
\Rightarrow &\mathds{E}\left[\text{e}^{-\int_0^t2L(s){\rm d}s}\left(\sum_{i=1}^M \left\|\widehat{X}_t^{(i)}-X_t^{(i)}\right\|^2\right)\right]\\
&\leq \mathds{E}\left[\sum_{i=1}^M \left\|\widehat{X}_0^{(i)}-X_0^{(i)}\right\|^2\right] + 2\mathds{E}\left[\int_0^t\text{e}^{-\int_0^s2L(r){\rm d}r}R(s){\rm d}s\right]\\
&\leq \mathds{E}\left[\sum_{i=1}^M \left\|\widehat{X}_0^{(i)}-X_0^{(i)}\right\|^2\right] + 2\int_0^t\mathds{E}\left[R(s)\right]{\rm d}s.
\end{aligned}
\end{equation}
Now use (\ref{firstest}) as well as Appendix \ref{appendixVkClassical} and Remark \ref{RemarkVClassical} for
\begin{equation}
\mathcal{V}_t\left\|K_{\nu_{+}(t)}\right\|^2 \leq \frac{1}{2}\left(\mathcal{V}_t\right)^2 + \frac{1}{2}\left\|K_{\nu_{+}(t)}\right\|^4
\end{equation}
to deduce that it holds
\begin{equation}
\sup_{t\in [0,T]}\mathds{E}\left[R(t)\right] \in O\left(h^{2\gamma}\right).
\end{equation}
By the assumption on the initial conditions this then yields
\begin{equation}
\sup_{t\in[0,T]} \mathds{E}\left[\text{e}^{-\int_0^t2L(s){\rm d}s}\left(\sum_{i=1}^M \left\|\widehat{X}_t^{(i)}-X_t^{(i)}\right\|^2\right)\right] \in O\left(h^{2\gamma}\right)
\end{equation}
and hence the claim
\begin{equation}
\sup_{t\in [0,T]} \mathds{E}\left[\text{e}^{-\int_0^t2L(s)ds}\left(\sum_{i=1}^M \left\|X_{\eta(t)}^{(i),a}-X_t^{(i)}\right\|^2\right)\right] \in O\left(h^{2\gamma}\right).
\end{equation}

\subsection{Proof of Theorem \ref{ModifiedLimit}}

\noindent
First of all observe from (\ref{modifiedSemimartingale})
\begin{equation}
\begin{aligned}
&X_{\eta(t)}^{(i)}-X_t^{(i)}\\
&= - \int_{\eta(t)}^t f\left(X_s^{(i)}\right)+\frac{1}{2}QP_s^{-1}\left(X_s^{(i)}-\bar{x}_s\right)- \frac{1}{2}\frac{1}{M-1}E_s\mathcal{G}_s^TC^{-1}\left(g\left(X_s^{(i)}\right)+\bar{g}_s\right){\rm d}s\\
&\hspace{0.5cm} - \int_{\eta(t)}^t \frac{1}{M-1}E_s\mathcal{G}_s^TC^{-1}{\rm d}Y_s\\
&= - \int_{\eta(t)}^t f\left(X_s^{(i)}\right)+\frac{1}{2}QP_s^{-1}\left(X_s^{(i)}-\bar{x}_s\right)+\frac{1}{M-1}E_s\mathcal{G}_s^TC^{-1}g\left(X_s^{\text{ref}}\right)\\
&\hspace{1.5cm}- \frac{1}{2}\frac{1}{M-1}E_s\mathcal{G}_s^TC^{-1}\left(g\left(X_s^{(i)}\right)+\bar{g}_s\right){\rm d}s\\
&\hspace{0.5cm} - \int_{\eta(t)}^t \frac{1}{M-1}E_s\mathcal{G}_s^TC^{-\frac{1}{2}}{\rm d}V_s.
\end{aligned}
\end{equation}
With this we obtain
\begin{equation}
\begin{aligned}
&\frac{1}{M-1}\sum_{i=1}^M \left\|X_{\eta(t)}^{(i)} - X_t^{(i)}\right\|^2 \\
&\leq \frac{2}{M-1}\sum_{i=1}^M \left\|\int_{\eta(t)}^t f\left(X_s^{(i)}\right)+\frac{1}{2}QP_s^{-1}\left(X_s^{(i)}-\bar{x}_s\right)+\frac{1}{M-1}E_s\mathcal{G}_s^TC^{-1}g\left(X_s^{\text{ref}}\right)\right.\\
&\hspace{3cm}\left.- \frac{1}{2}\frac{1}{M-1}E_s\mathcal{G}_s^TC^{-1}\left(g\left(X_s^{(i)}\right)+\bar{g}_s\right){\rm d}s\right\|^2\\
&\hspace{0.5cm}+\frac{2M}{M-1}\left\|\int_{\eta(t)}^t \frac{1}{M-1}E_s\mathcal{G}_s^TC^{-\frac{1}{2}}{\rm d}V_s\right\|^2\\
&=: (I) + (II).
\end{aligned}
\end{equation}
Using the estimate (\ref{third}) in combination with Appendix \ref{appendixVtmod} as well as boundedness of $f$ and $g$ we deduce
\begin{equation}
(I) \in O\left(h^2\right).
\end{equation}
Further we obtain that the process
\[\int_0^{\cdot} \frac{1}{M-1}E_s\mathcal{G}_s^TC^{\frac{1}{2}}{\rm d}V_s\]
is H\"older-continuous with coefficient $\rho < \frac{1}{2}$ (cf. \cite{roynette1993}). In total, this yields
\begin{equation}
\frac{1}{M-1}\sum_{i=1}^M \left\|X_{\eta(t)}^{(i)} - X_t^{(i)}\right\|^2 \in O\left(h^{2\rho}\right).
\end{equation}

\medskip
\noindent
Now let $1\leq k \leq L$, then by using the semimartingale representation of $Y$ and applying the Cauchy-Schwarz-inequality, we obtain

\begin{equation}
\begin{aligned}
&\frac{1}{M-1} \sum_{i=1}^M \left\|X_{t_k}^{(i),a} - X_{t_k}^{(i)}\right\|^2\\
&= \frac{1}{M-1}\sum_{i=1}^M \left\|X_0^{(i),a} - X_0^{(i)}+\int_0^{t_k} f\left(X_{\eta(s)}^{(i),a}\right)-f\left(X_s^{(i)}\right){\rm d}s\right.\\
&\hspace{2.6cm} + \int_0^{t_k} \frac{1}{2}Q\left(\left(P_{\nu(s)}^{a}\right)^{-1}\left(X_{\eta(s)}^{(i),a}-\bar{x}_{\nu(s)}^{a}\right)-P_s^{-1}\left(X_s^{(i)}-\bar{x}_s\right)\right){\rm d}s\\
&\hspace{2.6cm} + \int_0^{t_k}\left(K_{\nu_+(s)}-\frac{1}{M-1}E_s\mathcal{G}_s^TC^{-1}\right)g\left(X_s^{\text{ref}}\right){\rm d}s\\
&\hspace{2.6cm} + \int_0^{t_k}\left(K_{\nu_+(s)}-\frac{1}{M-1}E_s\mathcal{G}_s^TC^{-1}\right)C^{\frac{1}{2}}{\rm d}V_s\\
&\hspace{2.6cm} - \int_0^{t_k}\frac{1}{2}\left(K_{\nu_+(s)}\left(g\left(X_{\eta_+(s)}^{(i),f}\right)+\bar{g}_{\nu_+(s)}^{f}\right)\right.\\
&\hspace{5cm}\left.\left.-\frac{1}{M-1}E_s\mathcal{G}_s^TC^{-1}\left(g\left(X_s^{(i)}\right)+\bar{g}_s\right)\right){\rm d}s\right\|^2\\
&\leq 6 \left(\frac{1}{M-1}\sum_{i=1}^M \left\|X_0^{(i),a} - X_0^{(i)}\right\|^2 + (1) + (2) + (3) + (4) + (5)\right).
\end{aligned}
\end{equation}
Term $(1)$ can be estimated by using the Lipschitz-continuity of $f$.\\
\noindent
In $(2)$, the integrand can be decomposed in the following way:
\begin{equation}
\begin{aligned}
&\left(P_{\nu(s)}^{a}\right)^{-1}\left(X_{\eta(s)}^{(i),a}-\bar{x}_{\nu(s)}^{a}\right)-P_s^{-1}\left(X_s^{(i)}-\bar{x}_s\right)\\
&= \left(\left(P_{\nu(s)}^{a}\right)^{-1}-P_s^{-1}\right)\left(X_{\eta(s)}^{(i),a}-\bar{x}_{\nu(s)}^{a}\right) +P_s^{-1}\left(\left(X_{\eta(s)}^{(i),a}-\bar{x}_{\nu(s)}^{a}\right)-\left(X_s^{(i)}-\bar{x}_s\right)\right)\\
&= -\left(P_{\nu(s)}^{a}\right)^{-1}\left(P_{\nu(s)}^{a}-P_s\right)P_s^{-1}\left(X_{\eta(s)}^{(i),a}-\bar{x}_{\nu(s)}^{a}\right) +P_s^{-1}\left(\left(X_{\eta(s)}^{(i),a}-\bar{x}_{\nu(s)}^{a}\right)-\left(X_s^{(i)}-\bar{x}_s\right)\right).
\end{aligned}
\end{equation}
It holds with $\mathcal{V}^{a}$ as defined in (\ref{VkaWk})
\begin{equation}
\left\|P_{\nu(s)}^{a}-P_s\right\|^2 \leq 8 \left(\mathcal{V}_{\nu(s)}^{a} + \mathcal{V}_s\right)\left(\frac{1}{M-1}\sum_{i=1}^M\left\|X_{\nu(s)}^{(i),a}-X_s^{(i)}\right\|^2\right)
\end{equation}
which yields by Appendix \ref{appendixVt} and Appendix \ref{appendixVkModified}
\begin{equation}
(2) \leq \frac{2t_k\|Q\|^2}{\left(\lambda_T^{*}\right)^2}\left(\frac{v_{T}^{a,*}(v_{T}^{a,*}+v_T^*)}{(p_T^*)^2}+1\right)\int_0^{t_k} \frac{1}{M-1}\sum_{i=1}^M \left\|X_{\eta(s)}^{(i),a}-X_s^{(i)}\right\|^2{\rm d}s.
\end{equation}

\medskip
\noindent
Term (3), (4) and (5) can be estimated via Lemma \ref{coundCommonCoeff} where in (4) we apply the Burkholder-Davis-Gundy inequality to obtain for any $t \in [0,T]$

\begin{equation}
\begin{aligned}
&\mathds{E}\left[\left\|\sup_{s\in [0,t]} \int_0^s \left(K_{\nu_+(r)}-\frac{1}{M-1}E_r\mathcal{G}_r^TC^{-1}\right)C^{\frac{1}{2}}{\rm d}V_r\right\|^2\right]\\
&\leq C_{BDG} \left\|C^{\frac{1}{2}}\right\|^2\int_0^t \mathds{E}\left[\left\|K_{\nu_+(s)}-\frac{1}{M-1}E_s\mathcal{G}_s^TC^{-1}\right\|^2\right]{\rm d}s\\
&\leq 16C_{BDG}\left\|C^{-1}\right\|^2 \left\|C^{\frac{1}{2}}\right\|^2\\
&\hspace{0.5cm}\left(h^2 t\frac{8\|g\|_{\infty}^6\left\|C^{-1}\right\|^2v_{T}^{f,*}M^3}{(M-1)^3} + \|g\|_{\text{Lip}}^2\left(v_{T}^{f,*}+v_T^*\right)\int_0^t\mathds{E}\left[\frac{1}{M-1}\sum_{i=1}^M \left\|X_{\eta_+(s)}^{(i),f}-X_s^{(i)}\right\|^2\right]{\rm d}s\right).
\end{aligned}
\end{equation}
\noindent
Further it holds by the forecast step (\ref{centerModF})
\begin{equation}
\begin{aligned}
&\frac{1}{M-1}\sum_{i=1}^M \left\|X_{\eta_+(s)}^{(i),f}-X_s^{(i)}\right\|^2\\
&\leq 3\left(\frac{1}{M-1}\sum_{i=1}^M\left\|X_{\eta(s)}^{(i),a}-X_s^{(i)}\right\|^2\right) + 3h^2\left(\frac{\|f\|_{\infty}^2M}{M-1}+\frac{\|Q\|^2}{4(p_T^*)^2}v_{T}^{a,*}\right)\\
&=: 3\left(\frac{1}{M-1}\sum_{i=1}^M\left\|X_{\eta(s)}^{(i),a}-X_s^{(i)}\right\|^2\right) + 3h^2r^*.
\end{aligned}
\end{equation}
In total we obtain
\begin{equation}
\begin{aligned}
&\mathds{E}\left[\sup_{t\in[0,T]}\frac{1}{M-1}\sum_{i=1}^M \left\|X_{\eta(t)}^{(i),a}-X_{\eta(t)}^{(i)}\right\|^2\right]\\
&\leq 6 \mathds{E}\left[\frac{1}{M-1}\sum_{i=1}^M \left\|X_0^{(i),a}-X_0^{(i)}\right\|^2\right]\\
&\hspace{0.5cm} + L_T \int_0^T \mathds{E}\left[\sup_{r \in [0,s]}\frac{1}{M-1}\sum_{i=1}^M \left\|X_{\eta(r)}^{(i),a} - X_r^{(i)}\right\|^2\right]{\rm d}s + R_T h^2
\end{aligned}
\end{equation}
where
\begin{equation}
\begin{aligned}
L_T &:= 18\|g\|_{\text{Lip}}^2\left(v_{T}^{f,*}+v_T^*\right)+6T\left(\|f\|_{\text{Lip}}^2 + \frac{2\|Q\|^2}{\left(\lambda_T^{*}\right)^2}\left(\frac{v_{T}^{a,*}(v_{T}^{a,*}+v_T^*)}{(p_T^*)^2}+1\right) \right.\\
&\hspace{6cm}\left.+ 24\frac{\|g\|_{\infty}^2\left\|C^{-1}\right\|^2M}{M-1}\|g\|_{\text{Lip}}^2\left(6v_{T}^{f,*}+7v_T^*\right)\right),
\end{aligned}
\end{equation}

\begin{equation}
\begin{aligned}
R_T &:= 48h^2 T\frac{\left\|C^{-1}\right\|^2 M}{M-1}\\
&\hspace{0.5cm}\left(\left(3T\|g\|_{\infty}^2+C_{BDG}\left\|C^{\frac{1}{2}}\right\|^2\right)\left(\frac{16\|g\|_{\infty}^6\left\|C^{-1}\right\|^2M^3}{(M-1)^3}v_{T}^{f,*} + 6\|g\|_{\text{Lip}}^2\left(v_{T}^{f,*}+v_T^*\right)r^*\right)\right.\\
&\hspace{1cm}\left.+ 3T\|g\|_{\infty}^2\|g\|_{\text{Lip}}^2v_T^*r^*\right).
\end{aligned}
\end{equation}
Now since by the initial discussion of this subsection and the Burkholder-Davis-Gundy inequality it holds
\begin{equation}
\mathds{E}\left[\sup_{t\in[0,T]} \frac{1}{M-1}\sum_{i=1}^M\left\|X_{\eta(t)}^{(i)}-X_t^{(i)}\right\|^2\right] \in O(h),
\end{equation}
we obtain by assumption on the initial conditions and by a Gronwall argument a constant $\tilde{C}_T>0$ such that
\begin{equation}
\mathds{E}\left[\sup_{t\in[0,T]}\frac{1}{M-1}\sum_{i=1}^M \left\|X_{\eta(t)}^{(i),a}-X_t^{(i)}\right\|^2\right] \leq \tilde{C}_Th
\end{equation}
which concludes the proof.

\section{Conclusion and Outlook}

\noindent
The Ensemble Kalman Filter is a powerful tool in the field of data assimilation, and its numerics are widely explored when applied to a variety of high-dimensional models as such arising in the geosciences. Our understanding of its theoretical properties, however, is still rather limited. Some investigations in this direction used the continuous version of the Ensemble Kalman Filter in context of observations arriving with high-frequency, and present first results on stability and accuracy. In this work, we investigated how the discrete filtering scheme used in the numerics and the continuous filtering scheme used for the mathematical analysis relate to each other in that we conducted a continuous time limit argument. In the literature, first results on this can be found in \cite{kelly2014}, as well as \cite{schillings2017} and \cite{bloemker2018} in the context of inverse problems, but these lack a rigorous proof and effective rates for the limit. Similar to \cite{kelly2014}, we first considered the classical Ensemble Kalman Filter formulation which uses a stochastically perturbed observation ensemble, and our result reads as follows:

\smallskip
\noindent
Choosing the Euler-Maruyama approximation of the signal and the observation as their respective time-discretization, the expected ensemble-mean-square error of the corresponding Ensemble Kalman Filter formulation and the classical continuous version converges to zero uniformly in time with convergence rate given in Theorem \ref{ClassicalLimit}.

\smallskip
\noindent
In the attempt of proving better rates, the stochastic perturbation of the observations posed a serious problem. When omitting this perturbation step the convergence speed can be improved; according to \cite{burgers1998}, however, such a perturbation is necessary for the algorithm to give the correct statistics of the ensemble.

\smallskip
\noindent
These findings led us to consider alternative versions of the Ensemble Kalman Filter that do not rely on the stochastic perturbation step and still produce the correct statistics. Recently in \cite{deWiljes2017}, the authors analyzed a modified continuous Ensemble Kalman Filter scheme originating from deterministic transformations of the ensemble. The very promising stability and accuracy properties shown in \cite{deWiljes2017} proposed this formulation as an appropriate candidate for our investigations. We summarize our result to:

\smallskip
\noindent
Choosing the Euler-Maruyama approximation of the signal and the observation as their respective time-discretization, the ensemble-mean-square error of the corresponding modified Ensemble Kalman Filter formulation and the modified continuous version converges to zero locally uniformly in time in expectation with convergence rate given in Theorem \ref{ModifiedLimit}.

\smallskip
\noindent
One may further consider investigating the convergence of a mixture of both formulations as mentioned in \cite{bergemann2012} where the classical forecast step is combined with the modified update step.\\
The extension to unbounded observation maps $g$ remains an open problem at present. Work is in progress though on the linear case.

\medskip 
\noindent 
{\bf Acknowlegdement} This research has been funded by Deutsche Forschungsgemeinschaft (DFG) through grant CRC 1294 "Data Assimilation", Project (A02) "Long-time stability and accuracy of ensemble transform filter algorithms".

\appendix

\section{Control of the continuous-time processes}\label{appendixVt}

\subsection{Classical formulation}

\begin{lemma}\label{boundVtclassical}
$(\mathcal{V}_t)$ can be controlled $\omega$-wise locally in time $t$. Precisely, it holds
\begin{equation}
\sup_{t \in [0,T]}\mathcal{V}_t \leq e^{AT}\mathcal{V}_0 + \mathcal{C} + e^{AT} \sup_{t \in [0,T]} \int_0^t e^{-As}{\rm d}N_s
\end{equation}
for constants $A = A(f)$ and $\mathcal{C} = \mathcal{C}(Q,f,T)$, and a martingale $N$.
\end{lemma}

\begin{proof}
First observe the following: since $\left(W^{(i)}\right)_{i=1,...,M}$ and $\left(V^{(i)}\right)_{i=1,...,M}$ are independent standard Brownian motions, the quadratic variation for each $i=1,...,M$ reads
\begin{equation}\label{quadVariationNoiseW}
\left \langle W^{(i)} - \bar{w}\right \rangle_t = \frac{M-1}{M}t = \left \langle V^{(i)} - \bar{v}\right \rangle_t.
\end{equation}
\noindent
By the It\^{o} formula we thus obtain
\begin{equation}
\begin{aligned}
{\rm d} \mathcal{V}_t &= \frac{1}{M-1} \sum_{i=1}^M {\rm d}\left\|X_t^{(i)}-\bar{x}_t\right\|^2\\
&= \frac{1}{M-1} \sum_{i=1}^M 2 \left\langle X_t^{(i)}-\bar{x}_t, f\left(X_t^{(i)}\right)-\bar{f}_t \right\rangle {\rm d}t\\
&\hspace{2.5cm}+ 2 \left \langle X_t^{(i)} - \bar{x}_t, Q^{\frac{1}{2}}\left({\rm d}W_t^{(i)} - {\rm d}\bar{w}_t\right) \right \rangle\\
&\hspace{2.5cm}+ 2 \left \langle X_t^{(i)} - \bar{x}_t, \frac{1}{M-1} E_t \mathcal{G}_t^T C^{-\frac{1}{2}}\left({\rm d} V_t^{(i)} - {\rm d}\bar{v}_t\right) \right \rangle\\
&\hspace{2.5cm}- 2 \left \langle X_t^{(i)} - \bar{x}_t, \frac{1}{M-1} E_t \mathcal{G}_t^T C^{-1} \left(g\left(X_t^{(i)}\right) - \bar{g}_t\right) \right \rangle {\rm d}t\\
&\hspace{0.5cm} + \text{tr}(Q) {\rm d}t + \frac{1}{(M-1)^2} \text{tr}\left(E_t \mathcal{G}_t^T C^{-1} \mathcal{G}_t E_t^T \right) {\rm d}t.
\end{aligned}
\end{equation}
Observe now that
\begin{equation}\label{Observation1}
\sum_{i=1}^M \left \langle X_t^{(i)} - \bar{x}_t, f(\bar{x}_t) - \bar{f}_t \right \rangle = 0
\end{equation}
and thus
\begin{equation}\label{Observation1.2}
\sum_{i=1}^M \left \langle X_t^{(i)} - \bar{x}_t, f\left(X_t^{(i)}\right) - \bar{f}_t \right \rangle = \sum_{i=1}^M \left \langle X_t^{(i)} - \bar{x}_t, f\left(X_t^{(i)}\right) - f(\bar{x}_t) \right \rangle.
\end{equation}

\medskip
\noindent
Further it holds
\begin{equation}
\begin{aligned}
&\frac{1}{M-1}\sum_{i=1}^M \left \langle X_t^{(i)} - \bar{x}_t, \frac{1}{M-1} E_t \mathcal{G}_t^T C^{-1} \left(g\left(X_t^{(i)}\right) - \bar{g}_t\right) \right \rangle \\
&\hspace{1cm}=\frac{1}{M-1}\sum_{i=1}^M \frac{1}{M-1}\text{tr}\left(\left(X_t^{(i)}-\bar{x}_t\right)\left(g\left(X_t^{(i)}\right)-\bar{g}_t\right)^TC^{-1}\mathcal{G}_tE_t^T\right)\\
&\hspace{1cm}= \frac{1}{(M-1)^2}\text{tr}\left(E_t\mathcal{G}_t^TC^{-1}\mathcal{G}_tE_t^T\right) = \left\|\frac{1}{M-1} E_t \mathcal{G}_t^T C^{-\frac{1}{2}}\right\|_F^2 \geq 0.
\end{aligned}
\end{equation}
\noindent
Thus
\begin{equation}
\begin{aligned}
{\rm d} \mathcal{V}_t &=\frac{1}{M-1} \sum_{i=1}^M 2 \left\langle X_t^{(i)}-\bar{x}_t, f\left(X_t^{(i)}\right)-f\left(\bar{x}_t\right) \right\rangle {\rm d}t\\
&\hspace{2.5cm}+ 2 \left \langle X_t^{(i)} - \bar{x}_t, Q^{\frac{1}{2}}\left({\rm d}W_t^{(i)} - {\rm d}\bar{w}_t\right) \right \rangle\\
&\hspace{2.5cm}+ 2 \left \langle X_t^{(i)} - \bar{x}_t, \frac{1}{M-1} E_t \mathcal{G}_t^T C^{-\frac{1}{2}}\left({\rm d} V_t^{(i)} - {\rm d}\bar{v}_t\right) \right \rangle\\
&\hspace{0.5cm} + \text{tr}(Q) {\rm d}t -\left\|\frac{1}{M-1} E_t \mathcal{G}_t^T C^{-\frac{1}{2}}\right\|_F^2 {\rm d}t.
\end{aligned}
\end{equation}
Therefore with the martingale $N$ with
\begin{equation}
{\rm d}N_t =  \frac{2}{M-1} \sum_{i=1}^M \left \langle X_t^{(i)} - \bar{x}_t, Q^{\frac{1}{2}}{\rm d}W_t^{(i)} + \frac{1}{M-1}E_t\mathcal{G}_t^TC^{-\frac{1}{2}}{\rm d}V_t^{(i)}\right \rangle
\end{equation}
we estimate
\begin{equation}\label{inequalityVt}
{\rm d} \mathcal{V}_t \leq 2(Lf)_{+} \mathcal{V}_t {\rm d}t + \text{tr}(Q) {\rm d}t + {\rm d}N_t.
\end{equation}
\medskip
\noindent
Applying the It\^{o} product rule yields
\begin{equation}
\begin{aligned}
{\rm d}\left( e^{-2(Lf)_{+}t}\mathcal{V}_t\right) &\leq e^{-2(Lf)_{+}t}\left(\text{tr}(Q){\rm d}t + {\rm d}N_t\right)
\end{aligned}
\end{equation}
and therefore it holds
\begin{equation}\label{Vest}
\mathcal{V}_t \leq e^{At}\mathcal{V}_0 + \frac{\text{tr}(Q)}{2(Lf)_{+}}\left(e^{2(Lf)_{+}t}-1\right) + \int_0^t e^{2(Lf)_{+}(t-s)}{\rm d}N_s.
\end{equation}
\noindent
Note that
\[\left( \int_0^t e^{-2(Lf)_{+}s}{\rm d}N_s\right)_{t \geq 0}\]
is a continuous martingale and thus locally bounded in time $t$. Hence we can control $\mathcal{V}_t$ $\omega$-wise locally in $t$, since for any $T > 0$ we get
\begin{equation}
\sup_{t \in [0,T]}\mathcal{V}_t \leq e^{2(Lf)_{+}T}\mathcal{V}_0 + \frac{\text{tr}(Q)}{2(Lf)_{+}}\left(e^{2(Lf)_{+}T}-1\right) + e^{2(Lf)_{+}T} \sup_{t \in [0,T]} \int_0^t e^{-2(Lf)_{+}s}{\rm d}N_s.
\end{equation}
\end{proof}
\begin{remark}\label{RemarkVClassical}
Since one can estimate by boundedness of $g$
\begin{equation*}
\frac{{\rm d}}{{\rm d}t} \langle N\rangle_t \leq \frac{4}{(M-1)^2}\left(\left\|Q^{\frac{1}{2}}\right\|^2 + \frac{4M\|g\|_{\infty}^2\left\|C^{-\frac{1}{2}}\right\|^2}{M-1}\mathcal{V}_t\right)\mathcal{V}_t,
\end{equation*}
one can further show by using Equation (\ref{Vest}) and a Gronwall argument that
\begin{equation}\label{Vsecondmom}
\sup_{t\in[0,T]}\mathds{E}\left[\left(\mathcal{V}_t\right)^2\right] < \infty.
\end{equation}
\end{remark}

\subsection{Modified formulation}\label{appendixVtmod}

\begin{lemma}\label{boundVtModified}
$\mathcal{V}_t$ is bounded for all $t \in [0,T]$, i.e.
\begin{equation}
\mathcal{V}_t \leq v_T^{*} := e^{2(Lf)_{+}T}\mathcal{V}_0 + \frac{{\rm tr}(Q)}{2(Lf)_{+}}\left(e^{2(Lf)_{+}T}-1\right).
\end{equation}
\end{lemma}
\begin{proof}
It holds
\begin{equation}
\begin{aligned}
\frac{1}{2} \frac{{\rm d}}{{\rm d}t} \mathcal{V}_t &= \frac{1}{M-1}\sum_{i=1}^M \left\langle X_t^{(i)} - \bar{x}_t, f\left(X_t^{(i)}\right) - \bar{f}_t\right \rangle\\
&\hspace{2.5cm}+\left \langle X_t^{(i)}-\bar{x}_t, \frac{1}{2}QP_t^{-1}\left(X_t^{(i)}-\bar{x}_t\right)\right \rangle\\
&\hspace{2.5cm}- \left \langle X_t^{(i)} - \bar{x}_t, \frac{1}{2}\frac{1}{M-1}E_t\mathcal{G}_t^TC^{-1}\left(g\left(X_t^{(i)}\right)-\bar{g}_t\right)\right\rangle\\
&= \frac{1}{M-1}\sum_{i=1}^M \left\langle X_t^{(i)} - \bar{x}_t, f\left(X_t^{(i)}\right) - f\left(\bar{x}_t\right)\right \rangle\\
&\hspace{0.5cm} + \frac{\text{tr}(Q)}{2} - \frac{1}{2}\frac{1}{(M-1)^2}\text{tr}\left( E_t\mathcal{G}_t^TC^{-1}\mathcal{G}_tE_t^T\right)\\
&= \frac{1}{M-1}\sum_{i=1}^M \left\langle X_t^{(i)} - \bar{x}_t, f\left(X_t^{(i)}\right) - f\left(\bar{x}_t\right)\right \rangle + \frac{\text{tr}(Q)}{2} - \frac{1}{2}\left\|\frac{1}{M-1} E_t\mathcal{G}_t^TC^{-\frac{1}{2}}\right\|_{F}^2\\
&\leq (Lf)_{+} \mathcal{V}_t + \frac{\text{tr}(Q)}{2}
\end{aligned}
\end{equation}
which yields the claim after applying a Gronwall argument.
\end{proof}

\medskip
\noindent
Recall the evolution equation of $P_t$
\begin{equation}\label{PtEvolution}
\begin{aligned}
\frac{{\rm d}}{{\rm d}t} P_t &= \frac{1}{M-1} \sum_{i=1}^M \left( \left(f\left(X_t^{(i)}\right)-\bar{f}_t\right)\left(X_t^{(i)}-\bar{x}_t\right)^T + \left(X_t^{(i)}-\bar{x}_t\right)\left(f\left(X_t^{(i)}\right)-\bar{f}_t\right)^T\right)\\
&\hspace{0.5cm} + Q - \frac{1}{(M-1)^2}E_t\mathcal{G}_t^T C^{-1} \mathcal{G}_t E_t^T.
\end{aligned}
\end{equation}
Further let $\lambda_t^{\min}$ denote the smallest eigenvalue of $P_t$ for each $t \in [0,T]$.

\begin{lemma}\label{boundLambdaModified}
If $\lambda_0^{\min}$ is bounded away from 0, then there exists a constant $\lambda_T^{*} > 0$ bounded away from 0 such that
\begin{equation}
\lambda_t^{\min} \geq \lambda_T^{*} \hspace{1cm} \forall t \in [0,T].
\end{equation}
\end{lemma}
\noindent
Note that, consequently, $\left\|P_t^{-1}\right\|$ is bounded from above by $\left(\lambda_T^{*}\right)^{-1}$.

\begin{proof}
First of all observe that for any $v \in \mathds{R}^{d}$ it holds
\begin{equation}\label{Pvv}
\left \langle P_t v, v\right\rangle = \frac{1}{M-1}\sum_{i=1}^M \left\langle X_t^{(i)}-\bar{x}_t,v\right\rangle^2.
\end{equation}

\medskip
\noindent
We proceed as in the proof of Lemma 6 in \cite{deWiljes2017}. Consider a diagonalization of $P_t$, i.e. orthogonal matrices $U_t$ and diagonal matrices $\Lambda_t$ such that
\begin{equation}
P_t = U_t^T \Lambda_t U_t.
\end{equation}
Then by Equation (\ref{PtEvolution})
\begin{equation}
\begin{aligned}
\frac{{\rm d}}{{\rm d}t} \Lambda_t &= \text{diag}\left(U_t\left(\frac{1}{M-1} \sum_{i=1}^M \left(f\left(X_t^{(i)}\right)-\bar{f}_t\right)\left(X_t^{(i)}-\bar{x}_t\right)^T \right.\right.\\
&\hspace{7cm}\left.\left.+ \left(X_t^{(i)}-\bar{x}_t\right)\left(f\left(X_t^{(i)}\right)-\bar{f}_t\right)^T\right)U_t^T\right)\\
&\hspace{0.5cm} + \text{diag}\left(U_tQU_t^T\right) - \text{diag}\left(U_t\left(\frac{1}{(M-1)^2}E_t\mathcal{G}_t^T C^{-1} \mathcal{G}_t E_t^T\right)U_t^T\right).
\end{aligned}
\end{equation}
Note that for any matrix $A\in \mathds{R}^{d\times d}$
\begin{equation}
\left(\text{diag}\left(U_tAU_t^T\right)\right)_{ll} = e_l^T U_t A U_t^T e_l
\end{equation}
and $\left\|U_t^T e_l\right\| = 1$ with $(e_l)$ the standard orthonormal basis in $\mathds{R}^d$.\\

\noindent
Let $v_l = U_t^T e_l$, then $\|v_l\| = 1$ and we estimate
\begin{equation}
\begin{aligned}
&\left| \left(\text{diag}\left(U_t \frac{1}{M-1} \sum_{i=1}^M \left(f\left(X_t^{(i)}\right) - \bar{f}_t\right)\left(X_t^{(i)}-\bar{x}_t\right)^T U_t^T\right)\right)_{ll}\right|\\
&=\frac{1}{M-1}\sum_{i=1}^M \left \langle f\left(X_t^{(i)}\right)-f\left(\bar{x}_t\right), v_l\right \rangle \left \langle X_t^{(i)}-\bar{x}_t,v_l\right \rangle\\
&\leq \left(\frac{1}{M-1}\sum_{i=1}^M \left \langle f\left(X_t^{(i)}\right)-f\left(\bar{x}_t\right),v_l\right \rangle^2\right)^{\frac{1}{2}}\left(\frac{1}{M-1}\sum_{i=1}^M \left\langle X_t^{(i)}-\bar{x}_t,v_l\right\rangle^2\right)^{\frac{1}{2}}\\
&\leq \|f\|_{\text{Lip}}\left(\mathcal{V}_t\right)^{\frac{1}{2}}\left \langle P_t v_l, v_l\right\rangle^{\frac{1}{2}}.
\end{aligned}
\end{equation}

\medskip
\noindent
Further
\begin{equation}
\begin{aligned}
&\left|\left(\text{diag}\left(U_t \frac{1}{(M-1)^2} E_t \mathcal{G}_t^T C^{-1} \mathcal{G}_t E_t^T U_t^T\right)\right)_{ll}\right| = \frac{1}{(M-1)^2} \left \langle C^{-1} \mathcal{G}_t E_t^T v_l, \mathcal{G}_t E_t^T v_l\right \rangle\\
&\leq \left\|C^{-1}\right\|\left\|\frac{1}{M-1} \mathcal{G}_tE_t^Tv_l\right\|^2\\
&\leq \left\|C^{-1}\right\|\left(\frac{1}{M-1}\sum_{i=1}^M \left\|g\left(X_t^{(i)}\right)-g\left(\bar{x}_t\right)\right\|^2\right)\left(\frac{1}{M-1}\sum_{i=1}^M \left\langle X_t^{(i)}-\bar{x}_t,v_l\right\rangle^2\right)\\
&\leq \left\|C^{-1}\right\|\|g\|_{\text{Lip}}^2\mathcal{V}_t\left\langle P_t v_l, v_l \right \rangle.
\end{aligned}
\end{equation}
Thus by Young's inequality we obtain for some $\epsilon > 0$
\begin{equation}
\begin{aligned}
\frac{{\rm d}}{{\rm d}t} \left(\Lambda_t\right)_{ll} &\geq -2\|f\|_{\text{Lip}}\sqrt{v_T^*}\left \langle P_t v_l, v_l\right \rangle^{\frac{1}{2}} + \lambda^{\min}(Q) - \|g\|_{\text{Lip}}^2\left\|C^{-1}\right\|v_T^*\left\langle P_t v_l, v_l\right \rangle\\
&\geq \lambda^{\min}(Q) - 2\|f\|_{\text{Lip}}^2v_T^*\epsilon - \left(\frac{2}{\epsilon}+\|g\|_{\text{Lip}}^2\left\|C^{-1}\right\|v_T^*\right)\left\langle P_t v_l, v_l\right \rangle.
\end{aligned}
\end{equation}

\medskip
\noindent
Choose $\epsilon >0$ small enough such that
\begin{equation}
q_{\epsilon} := \lambda^{\min}(Q) - 2\|f\|_{\text{Lip}}^2v_T^*\epsilon > 0.
\end{equation}
Now since $\left \langle P_t v_l, v_l \right \rangle = \left(\Lambda_t\right)_{ll}$, it holds for $l$ such that $\lambda_t^{\min} = \left(\Lambda_t\right)_{ll}$ and with
\begin{equation}
\alpha_{\epsilon} := \frac{2}{\epsilon}+\|g\|_{\text{Lip}}^2\left\|C^{-1}\right\|v_T^*
\end{equation}
that
\begin{equation}
\frac{{\rm d}}{{\rm d}t} \lambda_t^{\min} \geq q_{\epsilon} - \alpha_{\epsilon}\lambda_t^{\min},
\end{equation}
thus by a Gronwall argument there exists $\lambda_{T,\epsilon}^* >0$ such that if $\lambda_0^{\min} > 0$ it holds
\begin{equation}
\lambda_t^{\min} \geq \lambda_{T,\epsilon}^*.
\end{equation}
\end{proof}

\section{Control of the discrete-time processes}\label{appendixVk}

\subsection{Classical formulation}\label{appendixVkClassical}

\noindent
By Algorithm \ref{ClassicalAlgo} we obtain for any $k$

\begin{equation}\label{recursiveVf}
\begin{aligned}
\mathcal{V}_{k+1}^{(i),f} &= \frac{1}{M-1} \sum_{i=1}^M \left\|X_{t_{k+1}}^{(i),f} - \bar{x}_{k+1}^{f}\right\|^2\\
&= \frac{1}{M-1} \sum_{i=1}^M \left\|X_{t_k}^{(i),a} - \bar{x}_k^{a}\right\|^2 + 2 h \left \langle X_{t_k}^{(i),a}-\bar{x}_k^{a},f\left(X_{t_k}^{(i),a}\right)-f\left(\bar{x}_k^{a}\right)\right \rangle \\
&\hspace{2.5cm}+ 2\left \langle X_{t_k}^{(i),a} - \bar{x}_k^{a}, Q^{\frac{1}{2}}\left(\tilde{W}_{k+1}^{(i)} - \bar{w}_{k+1}\right)\right \rangle + h^2 \left\|f\left(X_{t_k}^{(i),a}\right) - \bar{f}_k^{a}\right\|^2 \\
&\hspace{2.5cm}+ 2h \left\langle f\left(X_{t_k}^{(i),a}\right) - f\left(\bar{x}_k^{a}\right), Q^{\frac{1}{2}}\left(\tilde{W}_{k+1}^{(i)}-\bar{w}_{k+1}\right)\right \rangle \\
&\hspace{2.5cm}+ \left\|Q^{\frac{1}{2}}\left(\tilde{W}_{k+1}^{(i)}-\bar{w}_{k+1}\right)\right\|^2.
\end{aligned}
\end{equation}
Thus for
\begin{equation}\label{VkaWk}
\mathcal{V}_k^{a} := \frac{1}{M-1}\sum_{i=1}^M \left\|X_{t_k}^{(i),a} - \bar{x}_k^{a}\right\|^2,\hspace{1cm} \mathds{W}_k := \frac{1}{M-1} \sum_{i=1}^M \left\|\tilde{W}_k^{(i)}-\bar{w}_k\right\|^2
\end{equation}
this gives by the Cauchy-Schwarz-inequality
\begin{equation}
\mathcal{V}_{k+1}^{f} \leq \left(2+2h(Lf)_{+} + 5h^2\|f\|_{\text{Lip}}^2\right)\mathcal{V}_k^{a} + 3\left\|Q^{\frac{1}{2}}\right\|^2\mathds{W}_{k+1}
\end{equation}
which yields with the estimate (\ref{K2})
\begin{equation}\label{Ksquared}
\begin{aligned}
\left\|K_{k+1}\right\|^2 &\leq \frac{4\|g\|_{\infty}^2\left\|C^{-1}\right\|^2M}{M-1}\left(\left(2+2h(Lf)_{+} + 5h^2\|f\|_{\text{Lip}}^2\right)\mathcal{V}_k^{a} + 3\left\|Q^{\frac{1}{2}}\right\|^2\mathds{W}_{k+1}\right)\\
&\leq: \tilde{K}_1\left(\tilde{K}_2(h)\mathcal{V}_k^{a}+\left\|Q^{\frac{1}{2}}\right\|^2\mathds{W}_{k+1}\right).
\end{aligned}
\end{equation}
By taking expectation in Equation (\ref{recursiveVf}), we furthermore estimate
\begin{equation}
\mathds{E}\left[\mathcal{V}_{k+1}^{f}\right] \leq \left(1+2h(Lf)_{+} + 4h^2\|f\|_{\text{Lip}}^2\right)\mathds{E}\left[\mathcal{V}_k^{a}\right] + \left\|Q^{\frac{1}{2}}\right\|^2\mathds{E}\left[\mathds{W}_{k+1}\right]
\end{equation}
by independence of $X_{t_k}^{(i),a}$ and $\tilde{W}_{k+1}^{(i)}$. Similar to the continuous case it holds
\begin{equation}
\tilde{W}_{k+1}^{(i)} - \bar{w}_{k+1} \sim \mathcal{N}\left(0, \frac{M-1}{M}h{\rm Id}\right)
\end{equation}
and thus $\mathds{E}\left[\mathds{W}_{k+1}\right] = h$, i.e.
\begin{equation}
\mathds{E}\left[\mathcal{V}_{k+1}^{f}\right] \leq \left(1+2h(Lf)_{+} + 4h^2\|f\|_{\text{Lip}}^2\right)\mathds{E}\left[\mathcal{V}_k^{a}\right] + \left\|Q^{\frac{1}{2}}\right\|^2h
\end{equation}
and
\begin{equation}
\begin{aligned}
\mathds{E}\left[\left\|K_{k+1}\right\|^2\right] &\leq \frac{4\|g\|_{\infty}^2\left\|C^{-1}\right\|^2M}{M-1}\left(\left(1+2h(Lf)_{+} + 4h^2\|f\|_{\text{Lip}}^2\right)\mathds{E}\left[\mathcal{V}_k^{a}\right] + \left\|Q^{\frac{1}{2}}\right\|^2h\right)\\
&=: \tilde{K}_3\left(\tilde{K}_4(h)\mathds{E}\left[\mathcal{V}_k^{a}\right] + \left\|Q^{\frac{1}{2}}\right\|^2h\right).
\end{aligned}
\end{equation}
\noindent
Thus $\mathds{E}\left[\left\|K_{k+1}\right\|^2\right]$ as well as $\mathds{E}\left[\left\|K_{k+1}\right\|^4\right]$ are locally bounded in $k$ according to the following lemma:
\begin{lemma}\label{MomentsVka}
If $\mathds{E}\left[\mathcal{V}_0^{a}\right] < \infty$ and $\mathds{E}\left[\left(\mathcal{V}_0^{a}\right)^2\right] < \infty$, then first and second moment of $\mathcal{V}_k^{a}$ are locally bounded in $k$, i.e. it for $h\ll 1$ there exist constants $\mathcal{C}_1$, $\mathcal{C}_2$, $\mathcal{C}_3$, and $\mathcal{C}_4$, all independent of $h$, such that
\begin{equation}
\mathds{E}\left[\mathcal{V}_k^{a}\right] \leq \text{e}^{\mathcal{C}_1T}\left(\mathds{E}\left[\mathcal{V}_0^{a}\right] + T\mathcal{C}_2\right)
\end{equation}
as well as
\begin{equation}
\mathds{E}\left[\left(\mathcal{V}_k^{a}\right)^2\right] \leq \mathcal{C}_3 \text{{\rm exp}}\left(T\mathcal{C}_4\right).
\end{equation}
\end{lemma}
\begin{proof} 
Using the Cauchy-Schwarz inequality as well as (\ref{Ksquared}) yields with Algorithm \ref{ClassicalAlgo} the following estimate
\begin{equation}\label{Vkrecursion}
\begin{aligned}
&\mathcal{V}_{k+1}^{a} - \mathcal{V}_k^{a} = \frac{1}{M-1} \sum_{i=1}^M \left\|X_{t_{k+1}}^{(i),a} - \bar{x}_{k+1}^{a}\right\|^2 - \left\|X_{t_k}^{(i),a}-\bar{x}_k^{a}\right\|^2\\
&= \frac{1}{M-1} \sum_{i=1}^M \left \langle \left(X_{t_{k+1}}^{(i),a} - \bar{x}_{k+1}^{a}\right) + \left(X_{t_k}^{(i),a} - \bar{x}_k^{a}\right), \left(X_{t_{k+1}}^{(i),a}-\bar{x}_{k+1}^{a}\right) - \left(X_{t_k}^{(i),a} - \bar{x}_k^{a}\right)\right \rangle\\
&\leq \left(h + 2h(Lf)_{+} + 5h^2\|f\|_{\text{Lip}}^2 + \frac{8h(1+h)M}{M-1}\|g\|_{\infty}^2\tilde{K}_1\tilde{K}_2(h)\right) \mathcal{V}_k^{a}\\
&\hspace{0.5cm} + \left(2+h+\frac{8h(1+h)M}{M-1}\|g\|_{\infty}^2\tilde{K}_1\right)\left\|Q^{\frac{1}{2}}\right\|^2 \mathds{W}_{k+1} + 2\left\|K_{k+1}\right\|^2\left\|C^{\frac{1}{2}}\right\|^2\mathds{V}_{k+1} + 2 N_{k+1}
\end{aligned}
\end{equation}
with
\begin{equation}
\mathds{V}_k := \frac{1}{M-1} \sum_{i=1}^M \left \| \tilde{V}_k^{(i)} - \bar{v}_k\right\|^2
\end{equation}
\noindent
and the martingale
\begin{equation}
\begin{aligned}
&N_{k+1} \\
&:= \frac{1}{M-1} \sum_{i=1}^M \left \langle X_{t_k}^{(i),a} - \bar{x}_k^{a}  + h \left( f\left(X_{t_k}^{(i),a}\right) - \bar{f}_k^{a} \right),  Q^{\frac{1}{2}}\left(\tilde{W}_{k+1}^{(i)} - \bar{w}_{k+1}\right)\right \rangle\\
&\hspace{2.5cm}  + \left \langle X_{t_k}^{(i),a} - \bar{x}_k^{a}  + h \left( f\left(X_{t_k}^{(i),a}\right) - \bar{f}_k^{a} \right)- h  K_{k+1} \left( g\left(X_{t_{k+1}}^{(i),f}\right) - \bar{g}_{k+1}^{f} \right),\right.\\
&\hspace{3.5cm}\left. K_{k+1} C^{\frac{1}{2}}\left( \tilde{V}_{k+1}^{(i)} - \bar{v}_{k+1} \right)\right \rangle\\
&\overset{d}{=} \frac{1}{M-1} \sum_{i=1}^M \int_{t_k}^{t_{k+1}}\left \langle X_{t_k}^{(i),a} - \bar{x}_k^{a}  + h \left( f\left(X_{t_k}^{(i),a}\right) - \bar{f}_k^{a} \right), Q^{\frac{1}{2}}{\rm d}W_s^{(i)}\right \rangle\\
&\hspace{2.15cm} +\int_{t_k}^{t_{k+1}}\left \langle X_{t_k}^{(i),a} - \bar{x}_k^{a}  + h \left( f\left(X_{t_k}^{(i),a}\right) - \bar{f}_{k}^{a} \right) - h  K_{k+1} \left( g\left(X_{t_{k+1}}^{(i),f}\right) - \bar{g}_{k+1}^{f} \right),\right.\\
&\hspace{4cm}\left. K_{k+1} C^{\frac{1}{2}}{\rm d}V_s^{(i)}\right \rangle.\\
\end{aligned}
\end{equation}

\medskip
\noindent
Thus by independence of $\mathds{V}_{k+1}$, $\mathds{W}_{k+1}$ and $\mathcal{V}_k^{a}$ for each $k$ and since $\mathds{E}\left[\mathds{W}_k\right] = h = \mathds{E}\left[\mathds{V}_k\right]$ for each $k$ we obtain
\begin{equation}
\mathds{E}\left[\mathcal{V}_{k+1}^{a}\right] \leq \left(1+hC_1(h)\right)\mathds{E}\left[\mathcal{V}_k^{a}\right] + hC_2(h)
\end{equation}
where
\begin{align}
C_1(h) &:= 1+ 2 (Lf)_{+} + 5h \|f\|_{\text{Lip}}^2 + \frac{8(1+h)M}{M-1} \|g\|_{\infty}^2\tilde{K}_1(h)\tilde{K}_2(h) + 2\left\|C^{\frac{1}{2}}\right\|^2\tilde{K}_3\tilde{K}_4(h),\\
C_2(h) &:= \left\|Q^{\frac{1}{2}}\right\|^2\left(2+h+\frac{8h(1+h)M}{M-1}\|g\|_{\infty}^2\tilde{K}_1 + 2\left\|C^{\frac{1}{2}}\right\|^2\tilde{K}_3(h)\right).
\end{align}
\noindent
Therefore by a Gronwall argument
\begin{equation}\label{gronwallEstVka}
\mathds{E}\left[\mathcal{V}_k^{a}\right] \leq \left(1+hC_1(h)\right)^k\mathds{E}\left[\mathcal{V}_0^{a}\right] + \sum_{j=0}^{k-1} \left(1+hC_1(h)\right)^j hC_2(h)
\end{equation}
and since $h = \frac{T}{L}$, we achieve for $h \ll 1$ the final estimate
\begin{equation}
\mathds{E}\left[\mathcal{V}_k^{a}\right] \leq e^{TC_1(1)}\left(\mathds{E}\left[\mathcal{V}_0^{a}\right] + TC_2(1)\right)
\end{equation}
for all $1\leq k \leq L$. Thus $\mathds{E}\left[\mathcal{V}_k^{a}\right]$ is bounded for all $1\leq k \leq L$ since $\mathds{E}\left[\mathcal{V}_0^{a}\right] < \infty$.\\

\medskip
\noindent
Now consider the result of applying a Gronwall argument to Equation (\ref{Vkrecursion}): introduce
\begin{align}
\tilde{C}_1(h) &:= 1+2(Lf)_{+}+5h\|f\|_{\text{Lip}}^2 + \frac{8(1+h)M}{M-1}\|g\|_{\infty}^2\tilde{K}_1\tilde{K}_2(h),\\
\tilde{C}_2(h) &:= \left\|Q^{\frac{1}{2}}\right\|^2\left(2+h+\frac{8h(1+h)M}{M-1}\|g\|_{\infty}^2\tilde{K}_1\right),
\end{align}
then
\begin{equation}
\mathcal{V}_k^{a} \leq \left(1+h\tilde{C}_1(h)\right)^k\mathcal{V}_0^{a} + \sum_{j=1}^k \left(1+h\tilde{C}_1(h)\right)^{k-j}\left(\tilde{C}_2(h)\mathds{W}_j + 2\left\|K_j\right\|^2\left\|C^{\frac{1}{2}}\right\|^2\mathds{V}_j + 2N_j\right).
\end{equation}
By the Cauchy-Schwarz inequality we obtain
\begin{equation}
\begin{aligned}
\left(\mathcal{V}_k^{a}\right)^2 &\leq 4\left(1+h\tilde{C}_1(h)\right)^{2k}\left(\mathcal{V}_0^{a}\right)^2+ 4 \left(\sum_{j=1}^k \left(1+h\tilde{C}_1(h)\right)^{k-j}\tilde{C}_2(h)\mathds{W}_j \right)^2\\
&\hspace{0.5cm} + 16 \left(\sum_{j=1}^k \left(1+h\tilde{C}_1(h)\right)^{k-j}\left\|K_j\right\|^2\left\|C^{\frac{1}{2}}\right\|^2\mathds{V}_j\right)^2+ 16 \left(\sum_{j=1}^k \left(1+h\tilde{C}_1(h)\right)^{k-j}N_j\right)^2\\
&=: (I) + (II) + (III) + (IV).
\end{aligned}
\end{equation}
Again we estimate for $h\ll 1$
\begin{equation}
(1+h\tilde{C}_1(h))^{k} \leq e^{T\tilde{C}_1(1)}
\end{equation}
which gives
\begin{equation}
\mathds{E}\left[(I)\right] \leq e^{2T\tilde{C}_1(1)}\mathds{E}\left[\left(\mathcal{V}_0^{a}\right)^2\right].
\end{equation}
Since $\tilde{W}_k^{(i)} \sim \mathcal{N}(0, h{\rm Id})$ this yields by (\ref{quadVariationNoiseW})
\begin{equation}
\mathds{E}\left[\left\|\tilde{W}_k^{(i)}-\bar{w}_k\right\|^4\right] = 3\left(\frac{M-1}{M}h\right)^2
\end{equation}
and thus $\mathds{E}\left[\left(\mathds{W}_k\right)^2\right] \leq 3h^2$, therefore we obtain
\begin{equation}
\mathds{E}\left[(II)\right] \leq 12The^{2T\tilde{C}_1(1)}\tilde{C}_2(1)^2.
\end{equation}
For $(III)$ we use that also $\mathds{E}\left[\left(\mathds{V}_k\right)^2\right] \leq 3h^2$ and by (\ref{Ksquared})
\begin{equation}
\left\|K_j\right\|^4 \leq 2\tilde{K}_1^2\left(\tilde{K}_2(h)^2\left(\mathcal{V}_{j-1}^{a}\right)^2 + \left\|Q^{\frac{1}{2}}\right\|^4\left(\mathds{W}_j\right)^2\right).
\end{equation}
Thus by independence of $K_j$ and $\mathds{V}_j$ for any $j$ this yields
\begin{equation}
\begin{aligned}
\mathds{E}\left[(III)\right] &\leq 96Th\left\|C^{\frac{1}{2}}\right\|^4e^{2T\tilde{C}_1(1)}\tilde{K}_1^2\left(\tilde{K}_2(h)^2\sum_{j=1}^k \mathds{E}\left[\left(\mathcal{V}_{j-1}^{a}\right)^2\right]+3Th\left\|Q^{\frac{1}{2}}\right\|^4\right).
\end{aligned}
\end{equation}
Now on $(IV)$: recall that $\left(\tilde{W}_k^{(i)}\right)$ and $\left(\tilde{V}_k^{(i)}\right)$ are i.i.d. sequences independent of each other. Then by independence of Brownian increments, the It\^{o} isometry as well as (\ref{Ksquared}) it holds
\begin{equation}
\begin{aligned}
&\mathds{E}\left[\left(\sum_{j=1}^k N_j\right)^2\right] = \sum_{j=1}^k \mathds{E}\left[\left(N_j\right)^2\right]\\
&= \sum_{j=1}^k \frac{1}{(M-1)^2} \\
&\hspace{1cm} \sum_{i=1}^M \mathds{E}\left[\left(\int_0^{t_j} \left \langle Q^{\frac{T}{2}}\left( X_{t_{j-1}}^{(i),a} - \bar{x}_{j-1}^{a}  + h \left( f\left(X_{t_{j-1}}^{(i),a}\right) - \bar{f}_{j-1}^{a} \right)\right)\mathds{1}_{\{s\in[t_{j-1},t_j]\}}, {\rm d}W_s^{(i)}\right \rangle\right)^2\right]\\
&\hspace{1.5cm}+ \mathds{E}\left[\left(\int_0^{t_j} \left \langle C^{\frac{T}{2}}K_j^T\left(X_{t_{j-1}}^{(i),a} - \bar{x}_{j-1}^{a}  + h \left( f\left(X_{t_{j-1}}^{(i),a}\right) - \bar{f}_{j-1}^{a} \right)\right.\right.\right.\right.\\
&\hspace{6.5cm}\left.\left.\left.\left. - h  K_{j} \left( g\left(X_{t_{j}}^{(i),f}\right) - \bar{g}_{j}^{f} \right)\right)\mathds{1}_{\{s\in[t_{j-1},t_j]\}},{\rm d}V_s^{(i)}\right \rangle\right)^2\right]\\
\end{aligned}
\end{equation}
\begin{equation*}
\begin{aligned}
&\leq \sum_{j=1}^k\frac{h}{M-1}\left(\left\|Q^{\frac{1}{2}}\right\|^2\left(1+2h(Lf)_{+}+4h^2\|f\|_{\text{Lip}}^2\right)\mathds{E}\left[\mathcal{V}_{j-1}^{a}\right]\right.\\
&\hspace{3cm}\left. + 2\left\|C^{\frac{1}{2}}\right\|^2\left(\left(1+2h(Lf)_{+}+4h^2\|f\|_{\text{Lip}}^2\right)\mathds{E}\left[\left\|K_j\right\|^2\mathcal{V}_{j-1}^{a}\right]\right.\right.\\
&\hspace{9cm}\left.\left. + 4h^2\|g\|_{\infty}^2\frac{M}{M-1}\mathds{E}\left[\left\|K_j\right\|^4\right]\right)\right)\\
&\leq \sum_{j=0}^{k-1}\frac{h}{M-1} \\
&\hspace{1.25cm} \left(2\left\|C^{\frac{1}{2}}\right\|^2\tilde{K}_1^2\tilde{K}_2(h)^2\left(1+2h(Lf)_{+}+4h^2\|f\|_{\text{Lip}}^2+\frac{8h^2\|g\|_{\infty}^2M}{M-1}\right)\mathds{E}\left[\left(\mathcal{V}_j^{a}\right)^2\right]\right.\\
&\hspace{1.5cm} + \left\|Q^{\frac{1}{2}}\right\|^2\left(1+2h(Lf)_{+}+4h^2\|f\|_{\text{Lip}}^2\right)\left(1+2\left\|C^{\frac{1}{2}}\right\|^2\tilde{K}_1h\right)\\
&\hspace{1.5cm}\left. + 48h^4\left\|C^{\frac{1}{2}}\right\|^2\|g\|_{\infty}^2\tilde{K}_1\left\|Q^{\frac{1}{2}}\right\|^4\frac{M}{M-1}\right).
\end{aligned}
\end{equation*}
Thus by the Cauchy-Schwarz inequality we obtain an estimate of the form
\begin{equation}
\mathds{E}\left[\left(\mathcal{V}_k^{a}\right)^2\right] \leq \widehat{C}_1(h) + h\widehat{C}_2(h)\sum_{j=0}^{k-1} \mathds{E}\left[\left(\mathcal{V}_j^{a}\right)^2\right]
\end{equation}
which for $h\ll 1$ and by a Gronwall argument leads to
\begin{equation}
\mathds{E}\left[\left(\mathcal{V}_k^{a}\right)^2\right] \leq \widehat{C}_1(1)\text{exp}\left(T\widehat{C}_2(1)\right)
\end{equation}
and this concludes the proof.
\end{proof}

\subsection{Modified formulation}\label{appendixVkModified}

\begin{lemma}\label{boundVkaModified}
If for the smallest eigenvalue of $P_0^{a}$ it holds
\begin{equation}
\lambda^{\min}\left(P_0^{a}\right) > 0,
\end{equation}
then there exists a constant $p_T^{*} > 0$ such that
\begin{equation}
\left\|\left(P_k^{a}\right)^{-1}\right\| \leq \frac{1}{p_T^{*}}
\end{equation}
and constants $v_T^{a,*}, v_T^{f,*}>0$ (depending on $\lambda^{\min}\left(P_0^{a}\right)$ but independent of $h$) such that
\begin{equation}
\mathcal{V}_k^{a} \leq v_T^{a,*} \text{ and } \mathcal{V}_k^{f} \leq v_T^{f,*} \hspace{1cm}\text{(see (\ref{boundVka}) and (\ref{boundVkf}))}
\end{equation}
for all $1 \leq k \leq L$.
\end{lemma}

\begin{proof}
First of all note that $\left\|\left(P_k^{a}\right)^{-1}\right\| = \left(\lambda^{\min}\left(P_k^{a}\right)\right)^{-1}$ where $\lambda^{\min}\left(P_k^{a}\right)$, the smallest eigenvalue of $P_k^{a}$, satisfies
\begin{equation}
\lambda^{\min}\left(P_k^{a}\right) = \inf_{v,\|v\|=1} \left \langle P_k^{a} v, v \right \rangle.
\end{equation}

\noindent
Thus let $v \in \mathds{R}^d$ with $\|v\|=1$. By the following recursive formula
\begin{equation}
\begin{aligned}
P_{k+1}^{f} &= \frac{1}{M-1} \sum_{i=1}^M \left(X_{t_{k+1}}^{(i),f} - \bar{x}_{k+1}^{f}\right)\left(X_{t_{k+1}}^{(i),f} - \bar{x}_{k+1}^{f}\right)^T\\
&= P_k^{a}+ h \left(\frac{1}{M-1}\sum_{i=1}^M \left(X_{t_k}^{(i),a} - \bar{x}_k^{a}\right)\left(f\left(X_{t_k}^{(i),a}\right) - f\left(\bar{x}_k^{a}\right)\right)^T\right.\\
&\hspace{4cm}\left. + \left(f\left(X_{t_k}^{(i),a}\right)-f\left(\bar{x}_k^{a}\right)\right)\left(X_{t_k}^{(i),a}-\bar{x}_k^{a}\right)^T\right) + hQ\\
&\hspace{0.5cm} + h^2 \frac{1}{M-1} \sum_{i=1}^M \left(f\left(X_{t_k}^{(i),a}\right) - \bar{f}_k^{a} + \frac{1}{2}Q\left(P_k^{a}\right)^{-1}\left(X_{t_k}^{(i),a}-\bar{x}_k^{a}\right)\right)\\
&\hspace{4.5cm}\left(f\left(X_{t_k}^{(i),a}\right) - \bar{f}_k^{a} + \frac{1}{2}Q\left(P_k^{a}\right)^{-1}\left(X_{t_k}^{(i),a}-\bar{x}_k^{a}\right)\right)^T,
\end{aligned}
\end{equation}
we obtain by the Young's inequality for some $\epsilon >0$
\begin{equation}\label{Pf}
\begin{aligned}
&\left \langle P_{k+1}^{f} v, v\right \rangle\\
&\geq \left \langle P_k^{a} v, v\right \rangle + 2h \frac{1}{M-1} \sum_{i=1}^M \left \langle f\left(X_{t_k}^{(i),a}\right) - f\left(\bar{x}_k^{a}\right),v\right \rangle \left \langle X_{t_k}^{(i),a} - \bar{x}_k^{a}, v \right \rangle + h \left \langle Q v, v\right \rangle\\
&\geq \left \langle P_k^{a} v, v \right \rangle - 2h \|f\|_{\text{Lip}} \left(\mathcal{V}_k^{a}\right)^{\frac{1}{2}} \left(\left\langle P_k^{a}v,v\right \rangle\right)^{\frac{1}{2}} + h \left \langle Qv,v\right \rangle\\
&\geq \left(1-\frac{h}{\epsilon}\right) \left\langle P_k^{a} v, v \right \rangle + h \left(\lambda^{\min}(Q) - \epsilon\|f\|_{\text{Lip}}^2\mathcal{V}_k^{a}\right).
\end{aligned}
\end{equation}

\medskip
\noindent
By the specific structure of the Kalman gain matrix we obtain the following recursive form for the analyzed ensemble covariance matrix
\begin{equation}
\begin{aligned}
P_{k+1}^{a} &=\frac{1}{M-1} \sum_{i=1}^M \left(X_{t_{k+1}}^{(i),a}- \bar{x}_{k+1}^{a}\right)\left(X_{t_{k+1}}^{(i),a}-\bar{x}_{k+1}^{a}\right)^T\\
&= P_{k+1}^{f} - h \frac{1}{M-1}\sum_{i=1}^M \left(X_{t_{k+1}}^{(i),f} - \bar{x}_{k+1}^{f}\right)\left(g\left(X_{t_{k+1}}^{(i),f}\right) - \bar{g}_{k+1}^{f}\right)^TK_{k+1}^T\\
&\hspace{0.5cm} +\frac{h^2}{4} \frac{1}{M-1}\sum_{i=1}^M K_{k+1}\left(g\left(X_{t_{k+1}}^{(i),f}\right) - \bar{g}_{k+1}^{f}\right)\left(g\left(X_{t_{k+1}}^{(i),f}\right)-\bar{g}_{k+1}^{f}\right)^TK_{k+1}^T
\end{aligned}
\end{equation}
\newpage
\noindent
thus
\begin{equation}\label{Pa}
\begin{aligned}
\left \langle P_{k+1}^{a} v, v\right \rangle &\geq \left \langle P_{k+1}^{f} v, v \right \rangle\\
&\hspace{0.5cm} -  \frac{h}{M-1}\sum_{i=1}^M \left\langle X_{t_{k+1}}^{(i),f} - \bar{x}_{k+1}^{f},v\right\rangle \left \langle K_{k+1}\left(g\left(X_{t_{k+1}}^{(i),f}\right) - \bar{g}_{k+1}^{f}\right),v\right \rangle\\
&\geq \left \langle P_{k+1}^{f} v, v \right \rangle\\
&\hspace{0.5cm} - h \left(\left \langle P_{k+1}^{f} v, v \right\rangle\right)^{\frac{1}{2}}\left(\frac{1}{M-1}\sum_{i=1}^M \left \langle K_{k+1}\left(g\left(X_{t_{k+1}}^{(i),f}\right) - \bar{g}_{k+1}^{f}\right),v\right \rangle^2\right)^{\frac{1}{2}}\\
&\geq \left(1-\frac{h}{\epsilon}\right)\left\langle P_{k+1}^{f} v, v\right \rangle - h\epsilon\frac{16\|g\|_{\infty}^4\left\|C^{-1}\right\|^2M^2}{(M-1)^2} \mathcal{V}_{k+1}^{f}.
\end{aligned}
\end{equation}

\medskip
\noindent
Further we obtain with Algorithm \ref{ModifiedAlgorithm}
\begin{equation}\label{Vkf}
\begin{aligned}
\mathcal{V}_{k+1}^{f} &= \frac{1}{M-1} \sum_{i=1}^M \left\|X_{t_{k+1}}^{(i),f} - \bar{x}_{k+1}^{f}\right\|^2\\
&\leq \left(1+2h(Lf)_{+}+4h^2\|f\|_{\text{Lip}}^2+h^2\|f\|_{\text{Lip}}\|Q\|\left\|\left(P_k^{a}\right)^{-1}\right\|\right) \mathcal{V}_k^{a}\\
&\hspace{0.5cm} + h\text{tr}(Q) + \frac{h^2}{4} \left\|\left(P_k^{a}\right)^{-1}\right\|\text{tr}\left(Q^2\right)
\end{aligned}
\end{equation}
and
\begin{equation}
\mathcal{V}_{k+1}^{a} = \frac{1}{M-1}\sum_{i=1}^M \left\|X_{t_{k+1}}^{(i),a} - \bar{x}_{k+1}^{a}\right\|^2\leq \left(1+h^2\frac{16\|g\|_{\infty}^4\left\|C^{-1}\right\|^2M^2}{(M-1)^2}\right) \mathcal{V}_{k+1}^{f}.
\end{equation}

\medskip
\noindent
Now assume that there exists a constant $D > 0$ such that
\begin{equation}\label{assumpLamMin}
\lambda^{\min}\left(P_k^{a}\right) \geq Dh.
\end{equation}
Then
\begin{equation}
\begin{aligned}
\mathcal{V}_{k+1}^{a} &\leq \text{exp}\left(hC^{(1)}(h,D)\right)\left(\mathcal{V}_k^{a}+hC^{(2)}(D)\right)
\end{aligned}
\end{equation}
where
\begin{align}
C^{(1)}(h,D) &:= h\frac{16\|g\|_{\infty}^4\left\|C^{-1}\right\|^2M^2}{(M-1)^2}+ 2(Lf)_{+}+4h\|f\|_{\text{Lip}}^2+\frac{\|f\|_{\text{Lip}}\|Q\|}{D},\\
C^{(2)}(D) &:= \text{tr}(Q) + \frac{\text{tr}\left(Q^2\right)}{4D}.
\end{align}
We make the ansatz
\begin{equation}\label{VaAnsatz}
\mathcal{V}_k^{a}\leq \sum_{l=0}^{k-1} \text{exp}\left(hC^{(1)}(h,D)(k-l)\right)hC^{(2)}(D).
\end{equation}
Then
\begin{equation}
\mathcal{V}_{k+1}^{a} \leq \sum_{l=0}^{k} \text{exp}\left(hC^{(1)}(h,D)((k+1)-l)\right)hC^{(2)}(D)
\end{equation}
and for $h \ll 1$
\begin{equation}\label{boundVka}
\mathcal{V}_{k+1}^{a} \leq C^{(2)}(D)\int_0^T \text{e}^{C^{(1)}(1,D)(T-s)}{\rm d}s =: v_{D,T}^{a,*}.
\end{equation}
Using (\ref{Vkf}) in combination with Assumption (\ref{assumpLamMin}) and the uniform bound on $\mathcal{V}_k^{a}$, we also obtain the existence of a constant $v_{D,T}^{f,*}$ such that
\begin{equation}\label{boundVkf}
\mathcal{V}_{k+1}^{f} \leq v_{D,T}^{f,*}.
\end{equation}
This yields with Equations (\ref{Pf}) and (\ref{Pa})
\begin{equation}
\begin{aligned}
\left \langle P_{k+1}^{a} v, v\right \rangle &\geq \left(1-\frac{h}{\epsilon}\right)^2 \left \langle P_k^{a} v,v\right\rangle\\
&\hspace{0.5cm} + h \left( \left(1-\frac{h}{\epsilon}\right)\lambda^{\min}(Q) - \epsilon\left(\|f\|_{\text{Lip}}^2v_{D,T}^{a,*} + \frac{16\|g\|_{\infty}^4\left\|C^{-1}\right\|^2M^2}{(M-1)^2} v_{D,T}^{f,*}\right)\right).
\end{aligned}
\end{equation}
Choose $\epsilon >0$ such that for some $\hat{h} > 0$
\begin{equation}
\frac{h}{\epsilon}\lambda^{\min}(Q) + \epsilon\left(\|f\|_{\text{Lip}}^2v_{D,T}^{a,*} + \frac{16\|g\|_{\infty}^4\left\|C^{-1}\right\|^2M^2}{(M-1)^2} v_{D,T}^{f,*}\right) \leq \frac{\lambda^{\min}(Q)}{2}
\end{equation}
for all $h < \hat{h}$, then
\begin{equation}
\begin{aligned}
\left \langle P_{k+1}^{a} v, v\right \rangle &\geq \left(1-\frac{h}{\epsilon}\right)^2\left \langle P_k^{a} v, v\right \rangle + \frac{h}{2}\lambda^{\min}(Q) \geq \left(\left(1-\frac{h}{\epsilon}\right)^2 + \frac{\lambda^{\min}(Q)}{2D}\right)Dh.
\end{aligned}
\end{equation}
Observe that
\begin{equation}
\lim_{h\longrightarrow 0} \left(1-\frac{h}{\epsilon}\right)^2 + \frac{\lambda^{\min}(Q)}{2D} = 1 + \frac{\lambda^{\min}(Q)}{2D} > 1,
\end{equation}
thus there exists an $\tilde{h} >0$ such that for all $h < \tilde{h}$ it holds
\begin{equation}
\left(1-\frac{h}{\epsilon}\right)^2 + \frac{\lambda^{\min}(Q)}{2D} = 1 + \frac{\lambda^{\min}(Q)}{2D} \geq 1
\end{equation}
and $\left \langle P_{k+1}^{a} v, v \right \rangle \geq Dh$. Since $v$ was chosen arbitrarily we may thus conclude
\begin{equation}
\lambda^{\min}\left(P_{k+1}^{a}\right) \geq Dh.
\end{equation}
In total this yields that if
\begin{equation}
\lambda^{\min}\left(P_0^{a}\right) \geq Dh,
\end{equation}
then
\begin{equation}
\lambda^{\min}\left(P_k^{a}\right) \geq Dh
\end{equation}
for all $1 \leq k \leq L$ and $h < h^*$ with $h^* >0$ small enough.

\medskip
\noindent
Further it holds
\begin{equation}
\left \langle P_{k+1}^{a} v, v\right \rangle\geq \left(1-\frac{h}{\epsilon}\right)^{2(k+1)}\left \langle P_0^{a}v,v\right \rangle + \frac{\lambda^{\min}(Q)}{2}\sum_{l=0}^k \left(1-\frac{h}{\epsilon}\right)^{2(k-l)}h.
\end{equation}
Observe that for some $h_0 < \epsilon$ there exists a $\beta > 0$ such that
\begin{equation}
1-\frac{h}{\epsilon} \geq \text{exp}\left(-\left(\frac{1}{\epsilon}+\beta\right)h\right)
\end{equation}
for all $h < h_0$. Thus
\begin{equation}
\begin{aligned}
&\left \langle P_{k+1}^{a} v, v \right \rangle\\
&\geq \text{exp}\left(-2\left(\frac{1}{\epsilon}+\beta\right)h(k+1)\right)\left \langle P_0^{a} v, v\right \rangle + \frac{\lambda^{\min}(Q)}{2}\sum_{l=0}^k \text{exp}\left(-2\left(\frac{1}{\epsilon}+\beta\right)(k-l)h\right)h\\
&\geq \text{exp}\left(-2\left(\frac{1}{\epsilon}+\beta\right)T\right)\left \langle P_0^{a}v,v\right \rangle + \frac{\lambda^{\min}(Q)}{2}\int_0^{kh}\text{exp}\left(-2\left(\frac{1}{\epsilon}+\beta\right)(kh-s)\right){\rm d}s\\
&\geq \text{exp}\left(-2\left(\frac{1}{\epsilon}+\beta\right)T\right)\left \langle P_0^{a}v,v\right \rangle.
\end{aligned}
\end{equation}
Therefore if $h$ is small enough and $\lambda^{\min}\left(P_0^{a}\right) > 0$, then
\begin{equation}
\lambda^{\min}\left(P_k^{a}\right) \geq \text{exp}\left(-2\left(\frac{1}{\epsilon}+\beta\right)T\right)\lambda^{\min}\left(P_0^{a}\right) > 0
\end{equation}
and hence there exists a constant $p_T^{*} > 0$ such that
\begin{equation}
\left\|\left(P_k^{a}\right)^{-1}\right\| \leq \frac{1}{p_T^{*}}
\end{equation}
for all $1 \leq k \leq L$.
\end{proof}
\end{document}